\documentclass[11pt]{article}
\usepackage[english]{babel}
\usepackage{amssymb,amsmath,amsthm}
\newtheorem{theorem}{Theorem}[section]
\newtheorem{lemma}[theorem]{Lemma}
\newtheorem{proposition}[theorem]{Proposition}
\newtheorem{corollary}[theorem]{Corollary}
\newtheorem{question}[theorem]{Question}
\newtheorem{example}[theorem]{Example}

{\theoremstyle{definition}}
{\theoremstyle{definition}\newtheorem{remark}[theorem]{Remark}}

\newtheorem*{thmb}{Theorem B}
\newtheorem*{thmgs}{Theorem GS}
\newtheorem*{thmpw}{Theorem PW}

\numberwithin{equation}{section}

\def\C{{\mathbb C}}
\def\K{{\mathbb K}}
\def\N{{\mathbb N}}
\def\Z{{\mathbb Z}}
\def\R{{\mathbb R}}
\def\D{{\mathbb D}}
\def\T{{\mathbb T}}
\def\M{{\mathbb M}}
\def\E{{\mathbb E}}
\def\F{{\cal F}}

\def\hh{{\mathcal H}}
\def\pp{{\mathbb P}}

\def\epsilon{\varepsilon}
\def\phi{\varphi}
\def\leq{\leqslant}
\def\geq{\geqslant}
\def\ker{{\tt ker}\,}
\def\spann{\hbox{\tt span}\,}

\def\bin#1#2{\left({{#1}\atop {#2}}\right)}
\def\ilim{\mathop{\hbox{$\underline{\hbox{\rm lim}}$}}\limits}
\def\slim{\mathop{\hbox{$\overline{\hbox{\rm lim}}$}}\limits}

\font\Goth=eufm10 scaled 1200
\def\uu{\hbox{{\Goth U}}}

\title{Orbits of coanalytic Toeplitz operators and weak hypercyclicity}

\author{Stanislav Shkarin}

\date{}

\begin{document}

\maketitle

\begin{abstract}
We prove a new criterion of weak hypercyclicity of a bounded linear operator on a Banach space. Applying this criterion, we solve few open questions. Namely, we show that if $G$ is a region of $\C$ bounded by a smooth Jordan curve $\Gamma$ such that $G$ does not meet the unit ball but $\Gamma$ intersects the unit circle in a non-trivial arc, then $M^*$ is a weakly hypercyclic operator on $H^2(G)$, where $M$ is the multiplication by the argument operator $Mf(z)=zf(z)$. We also prove that if $g$ is a non-constant function from the Hardy space $H^\infty(\D)$ on the unit disk $\D$ such that $g(\D)\cap\D=\varnothing$ and the set $\{z\in\C:|z|=1,\ |g(z)|=1\}$ is a subset of the unit circle $\T$ of positive Lebesgue measure, then the coanalytic Toeplitz operator $T^*_g$ on the Hardy space $H^2(\D)$ is weakly hypercyclic. On the contrary, if $g(\D)\cap\D=\varnothing$, $|g|>1$ almost everywhere on $\T$ and $\log(|g|-1)\in L^1(\T)$, then $T^*_g$ is not $1$-weakly hypercyclic and hence is not weakly hypercyclic (a bounded linear operator $T$  on a complex Banach space $X$ is called $n$-weakly hypercyclic if there is $x\in X$ such that for every surjective continuous linear operator $S:X\to \C^n$, the set $\{S(T^mx):m\in\N\}$ is dense in $\C^n$). The last result is based upon lower estimates of the norms of the members of orbits of a coanalytic Toeplitz operator. Finally, we show that there is a $1$-weakly hypercyclic operator on a Hilbert space, whose square is non-cyclic and prove that a Banach space operator is weakly hypercyclic if and only if it is $n$-weakly hypercyclic for every $n\in\N$.
\end{abstract}

\normalsize

\section{Introduction \label{s1}}

In this article, all vector spaces are assumed to be over the field $\K$ being either the field $\R$ of real numbers or the field $\C$ of complex numbers. As usual, $\N$ is the set of all positive integers, $\Z_+=\N\cup\{0\}$, $\T=\{z\in\C:|z|=1\}$ and $\D=\{z\in\C:|z|<1\}$. The symbol $\log$ always sands for the logarithm base $e$. The symbol $\lambda$ always denotes the normalized Lebesgue measure on $\T$. When we write, say $L^p(\T)$, we always assume $\T$ to be equipped with $\lambda$.
For a Banach space $X$, $L(X)$ stands for the  algebra of bounded linear operators on $X$, while $X^*$ is the space of continuous linear functionals on $X$. For $T\in L(X)$, its dual is denoted $T'$: $T'\in L(X^*)$, $T'f(x)=f(Tx)$ for every $f\in X^*$ and every $x\in X$. For a continuous linear operator $T$ on a Hilbert space $\hh$, the symbol $T^*$ denotes the Hilbert space adjoint of $T$: $T^*\in L(\hh)$, $\langle T^*x,y\rangle=\langle x,Ty\rangle$ for every $x,y\in\hh$, where $\langle \cdot,\cdot\rangle$ is the inner product of $\hh$.

\begin{remark}\label{rire}
Let $\hh$ be a Hilbert space and $T\in L(\hh)$. Then,
due to the Riesz theorem, $T'$ and $T^*$ are similar with the similarity provided by an $\R$-linear isometry $R$ between $\hh$ and $\hh^*$, which associates to $x\in\hh$ the functional $\langle \cdot,x\rangle$. Note that $R$ is conjugate linear in the case $\K=\C$. \end{remark}

For a self-adjoint operator $T$ on a Hilbert space $\hh$, we write $T\geq 0$ if $\langle Tx,x\rangle\geq 0$ for each $x\in\hh$ and we write $T>0$ if $\langle Tx,x\rangle>0$ for each non-zero $x\in\hh$. Equivalently, $T>0$ iff $T\geq 0$ and $\ker T=\{0\}$.
As usual, we write $T\geq S$ or $S\leq T$ is $T-S\geq 0$ and we write $T>S$ or $S<T$ if $T-S>0$. A continuous linear operator $T$ on a Hilbert space $\hh$ is called {\it hyponormal} if $T^*T\geq TT^*$. Similarly, $T$ is called {\it hypernormal} if $TT^*\geq T^*T$. Let $X$ be a Banach space, $T\in L(X)$ and $x\in X$. Recall that $x$ is called a {\it hypercyclic vector} (or a {\it norm hypercyclic vector}) for $T$ if the orbit
$$
O(T,x)=\{T^nx:n\in\Z_+\}
$$
is norm dense in $X$. Similarly, $x$ is called a {\it weakly hypercyclic vector} for $T$ if $O(T,x)$ is dense in $X$ with respect to the weak topology on $X$. Recently, Feldman \cite{fe} has introduced and studied the concept of an $n$-weakly hypercyclic vector. Namely, for $n\in\N$, $x$ is called an {\it $n$-weakly hypercyclic vector} for $T$ if for every continuous surjective linear map $S:X\to\K^n$, the set  $S(O(T,x))$ is dense in $\K^n$. The operator $T$ is called {\it hypercyclic} (respectively, {\it weakly hypercyclic} or {$n$-weakly hypercyclic}) if it possesses a hypercyclic (respectively, {\it weakly hypercyclic} or {$n$-weakly hypercyclic}) vector.

The concepts of (norm) supercyclicity, weak supercyclicity and $n$-weak supercyclicity are defined in exactly the same way: one has just to replace the orbit $O(T,x)$ by the projective orbit
$$
O_{\rm pr}(T,x)=\{zT^nx:n\in\Z_+,\,z\in\K\}.
$$
For further information on these and other concepts of linear dynamics we refer to the book \cite{bama-book} and references therein.

We prove the following two criteria of weak hypercyclicity of a bounded operator on a Banach space. Recall that if $T\in L(X)$ and $x\in X$, then a {\it backward $T$-orbit} of $x$ is a sequence $\{x_n\}_{n\in\Z_+}$ in $X$ such that $x_0=x$ and $Tx_{n+1}=x_n$ for every $n\in\Z_+$.

\begin{theorem}\label{qq}Let $X$ be a separable infinite dimensional Banach space and $T\in L(X)$. Assume that there is a $T$-invariant dense linear subspace $F$ of $X$ such that
\begin{itemize}\itemsep=-2pt
\item[\rm(B1)]the space $F$ carries an inner product
$\langle\cdot,\cdot\rangle$ such that the corresponding norm $\|x\|_0=\sqrt{\langle x,x\rangle}$ on $F$
defines a topology $($not necessarily strictly$)$ stronger than the one inherited from $X;$
\item[\rm(B2)]the continuous extension of the restriction $S=T\big|_F:F\to F$ to the $($Hilbert space$)$ completion of the normed space $(F,\|\cdot\|_0)$ is an isometry with no non-trivial finite dimensional invariant subspaces$;$
\item[\rm(B3)]every $x\in F$ has a backward $T$-orbit norm-convergent to $0$ in the Banach space $X.$
\end{itemize}
Then the operator $T$ is weakly hypercyclic.
\end{theorem}

\begin{theorem}\label{qq1}Let $X$ be a separable infinite dimensional Banach space and $T\in L(X)$. Assume that there is a $T$-invariant dense linear subspace $F$ of $X$ such that the condition {\rm(B1)} is satisfied and
\begin{itemize}\itemsep=-2pt
\item[\rm(B$2'$)]the restriction of $T$ to $F$
is an isometry on the normed space $(F,\|\cdot\|_0)$ and $\langle T^nx,y\rangle\to 0$ for every $x,y\in F;$
\item[\rm(B$3'$)]for every $x\in F$, there is a backward $T$-orbit of $x$ containing $0$ in its norm closure.
\end{itemize}
Then the operator $T$ is weakly hypercyclic.
\end{theorem}

\begin{remark}\label{re1} We say that a bounded linear operator $T$ on a normed space $X$ is
{\it weakly nullifying} if the sequence $\{T^nx\}_{n\in\N}$ weakly converges to $0$ for each $x\in X$. It is well-known and easy to see that every unitary operator with purely absolutely continuous spectrum is weakly nullifying.
Note that (B$2'$) is automatically satisfied if the continuous extension of the restriction $S=T\big|_F:F\to F$ to the $($Hilbert space$)$ completion of the normed space $(F,\|\cdot\|_0)$ is a weakly nullifying isometry.
\end{remark}

We start by illustrating the above criteria of weak hypercyclicity.
Note that the first ever example of a weakly hypercyclic operator, which is not norm hypercyclic, was obtained by Chan and Sanders \cite{cs}. Recall that if $1\leq p<\infty$ and $w=\{w_n\}_{n\in\Z}$ is a bounded sequence of non-zero scalars, then bilateral weighted shift $T_w\in L(\ell_p(\Z))$ is defined by the formula $(T_wx)_n=w_{n+1}x_{n+1}$. Equivalently, $T_we_n=w_ne_{n-1}$, where $\{e_n\}_{n\in\Z}$ is the standard Schauder basis in $\ell_p(\Z)$. Chan and Sanders \cite{cs} proved that the bilateral weighted shift $T_w$ with the weight sequence $w_n=1$ for $n\leq 0$ and $w_n=2$ for $n>0$ is weakly hypercyclic on $\ell_2(\Z)$. In order to illustrate how to apply Theorems~\ref{qq} and~\ref{qq1}, we show how they imply the main result of \cite{cs} and more.

\begin{example}\label{bws} Let $p\geq2$ and $w=\{w_n\}_{n\in\Z}$ be a sequence of positive numbers. Denote $r_0=1$, $r_n=w_1^{-1}\cdot{\dots}\cdot w_n^{-1}$ for $n>0$ and $r_n=w_{1+n}\cdot{\dots}\cdot w_0$ for $n<0$. Assume also that the sequence $r=\{r_n\}_{n\in\Z}$ is bounded and $\ilim\limits_{n\to\infty}r_n=0$. Then the bilateral weighted shift $T_w$ is a weakly hypercyclic operator on $\ell_p(\Z)$. If additionally, $\ilim\limits_{n\to\infty}r_{-n}>0$, then $T_w$ is not norm hypercyclic.
\end{example}

\begin{proof} It is easy to see that if $\ilim\limits_{n\to\infty}r_{-n}>0$, then $\inf\{\|T_w^nx\|:n\in\Z_+\}>0$ for every non-zero $x\in\ell_p(\Z)$ and therefore $T_w$ can not be norm
hypercyclic.

Since $r$ is bounded and $p\geq 2$, the space
$$
\textstyle G=\Bigl\{x\in\K^\Z:\sum\limits_{n=-\infty}^\infty \frac{|x_n|^2}{r_n^2}<\infty\Bigr\}\ \ \text{equipped with the inner product}\ \
\langle x,y\rangle = \sum\limits_{n=-\infty}^\infty \frac{x_n\overline{y_n}}{r_n^2}
$$
satisfies $G\subseteq\ell_2(\Z)\subseteq \ell_p(\Z)$ and the Hilbert space topology of $G$ is stronger than the one inherited from $\ell_p(\Z)$. Furthermore, since $T_w(r_ne_n)=r_{n-1}e_{n-1}$ for every $n\in\Z$, $G$ is $T$-invariant and the restriction $T\bigr|_G:G\to G$ as an operator on the Hilbert space $G$ is unitarily equivalent to the unweighted backward shift on $\ell_2(\Z)$. The latter is an isometry and is a unitary operator with purely absolutely continuous spectrum in the complex case. Hence $T\bigr|_G:G\to G$ as an operator on the Hilbert space $G$ is a weakly nullifying isometry.

Let $F=\spann\{e_n:n\in\Z\}$. Then $F$ is a dense in $\ell_p(\Z)$ $T$-invariant linear subspace of $G$. Clearly, every $x\in F$ has a unique backward $T_w$-orbit. Condition $\ilim\limits_{n\to\infty}r_n=0$ ensures that $0$ is in the $\ell_p$-norm closure of the backward $T_w$-orbit of any $x\in F$. Thus (B1), (B$2'$) and (B$3'$) are satisfied. Theorem~\ref{qq1} implies that $T_w$ is a weakly hypercyclic operator on $\ell_p(\Z)$.
\end{proof}

The above example represents a known result, see \cite{ss}. It is included merely as an illustration.

Feldman \cite{fe} has conjectured that if for a continuous linear operator $T$ on a separable Banach space $X$ every $x\in X$ has a norm convergent to zero backward orbit and every $x$ from a dense in $X$ subset have weakly convergent to $0$ forward orbit, then $T$ is weakly hypercyclic or at least 1-weakly hypercyclic. It is shown in \cite{ss} that every weakly hypercyclic bilateral weighted shift on $\ell_p(\Z)$ with $p<2$ is norm hypercyclic. Thus if we take the above mentioned weight $w=(\dots,1,1,1,2,2,2,\dots)$ of Chan and Sanders, the corresponding weighted shift $T_w$ is not weakly hypercyclic on $\ell_p(\Z)$ with $1<p<2$. On the other hand, $T_w$ is invertible, all its backward orbits norm-converge to $0$ and the forward orbits of elements of $\spann\{e_n:n\in\Z\}$ converge weakly to 0. Thus $T_w\in L(\ell_p(\Z))$ with $1<p<2$ provides a counterexample to the weak hypercyclicity part of the above conjecture \cite{fe}. The 1-weak hypercyclicity part of the conjecture as well as its weak hypercyclicity part for Hilbert space operators remain open.

Right here we include another example of an application of the above criteria. Feldman \cite{fe} conjectured that if $G$ is a region of $\C$ bounded by a smooth Jordan curve $\Gamma$ such that $G$ does not meet the unit ball but $\Gamma$ intersects the unit circle in a non-trivial arc, then $M^*$ is a 1-weakly hypercyclic operator on $H^2(G)$, where $M$ is the multiplication by the argument operator $Mf(z)=zf(z)$. We prove this conjecture by means of applying Theorem~\ref{qq}. Note that with just a small adjustment of the proof the smoothness condition can be significantly relaxed.

\begin{theorem}\label{mz} Let $\Gamma$ be a $C^1$-smooth closed path $(=$a homeomorphic image of $\T$ under a $C^1$-map from $\T$ into $\C$ with nowhere vanishing derivative$)$ encircling the bounded domain $G$. Assume also that $G\cap\D=\varnothing$ and that $\Gamma\cap\T$ contains a non-trivial open arc $A$ of $\T$. Then $M^*$ is a weakly hypercyclic operator on the Hardy space $H^2(G)$, where $M$ is the multiplication by the argument operator $Mf(z)=zf(z)$.
\end{theorem}

We will need the following very well-known fact. We reproduce its proof for the sake of convenience.

\begin{lemma}\label{pb} Let $X$ be a Banach space and $T\in L(X)$ be power bounded, that is, $c=\sup\{\|T^n\|:n\in\Z_+\}<\infty$. Then the linear subspace $E=\{x\in X:\|T^nx\|\to 0\}$ of $X$ is closed. In particular, if $\|T^nx\|\to 0$ for $x$ from a dense subset of $X$, then $\|T^nx\|\to 0$ for each $x\in X$.
\end{lemma}

\begin{proof}Consider the map $\Phi:X\to\ell_\infty(\Z_+)$, $\Phi(x)=\{\|T^nx\|\}_{n\in\Z_+}$. Since $T$ is power bounded, the map is well-defined and Lipschitz with the constant $c$ and therefore is continuous, where $\ell_\infty(\Z_+)$ is assumed to be equipped with its usual sup-norm. Then $E=\Phi^{-1}(c_0(\Z_+))$ is closed since $c_0(\Z_+)$ is closed in $\ell_\infty(\Z_+)$.
\end{proof}

\begin{proof}[Proof of Theorem~$\ref{mz}$] Obviously, $M$ is invertible: $M^{-1}f(z)=\frac{f(z)}{z}$. Hence its dual $M'$ is invertible and $(M')^{-1}=(M^{-1})'$. Furthermore, since $G\cap \D=\varnothing$, $\|(M')^{-1}\|=\|M^{-1}\|\leq 1$. In particular, $(M')^{-1}$ is power bounded. For each $w\in G$, the evaluation map $\delta_w(f)=f(w)$ is a continuous linear functional on $H^2(G)$. It is easy to see that $(M')^{-n}\delta_w=w^{-n}\delta_w$. Since $G\cap \D=\varnothing$, $|w|>1$ and therefore $\|(M')^{-n}\delta_w\|\to 0$. Hence $\|(M')^{-n}\psi\|\to 0$ for every $\psi\in E=\spann\{\delta_w:w\in G\}$. Since the functionals $\delta_w$ for $w\in G$ separate the points of the reflexive Banach space $H^2(G)$, $E$ is dense in $H^2(G)^*$. By Lemma~\ref{pb}, $\|(M')^{-n}\psi\|\to 0$ for every $\psi\in H^2(G)^*$. That is, every $\psi\in H^2(G)^*$ has a norm convergent to 0 backward $M'$-orbit.

Now it is easy to see that for every $\phi\in L^2(A,\lambda)$, $G_\phi(f)=\int_A f(z)\phi(z)\,d\lambda(z)$ is a continuous linear functional on $H^2(G)$, where $f(z)$ in the integral stands for the non-tangential boundary value of $f$ at $z$. Furthermore, the map $\Phi(\phi)=G_\phi$ is an injective continuous linear operator from $L^2(A,\lambda)$ to $H^2(G)^*$. Moreover, it is easy to check that $M'(G_\phi)=G_{S\phi}$, where $S\phi(z)=z\phi(z)$. Thus, $F=\Phi(L^2(A,\lambda))$ is an $M'$-invariant subspace of $H^2(G)^*$, the inner product $\langle\cdot,\cdot\rangle$ on $F$ transferred from $L^2(A,\lambda)$ by the operator $\Phi$ defines the Hilbert space topology stronger than the one inherited from the norm topology of $H^2(G)^*$ and the restriction of $T$ to $F$ is a unitary operator with purely absolutely continuous spectrum on the Hilbert space $F$ (it is unitarily equivalent to $S$). In particular, the last restriction has no non-trivial finite dimensional subspaces. Thus conditions (B1--B3) for $T=M'$ are satisfied. By Theorem~\ref{qq}, $M'$ is weakly hypercyclic.
According to Remark~\ref{rire}, $M^*$ and $M'$ are similar with the similarity provided by an $\R$-linear isometry. Hence $M^*$ is weakly hypercyclic as well.
\end{proof}

We would like to mention the following result of the opposite nature.

\begin{theorem}\label{norm1} Let $\hh$ be a complex Hilbert space and $T\in L(\hh)$ be such that there exists $S\in L(\hh)$ for which $S^2\neq 0$, $TS=ST$ and $T^*T\geq S^*S+I$. Then $T$ is not $1$-weakly hypercyclic.
\end{theorem}

Now we describe the main application of the above results. Recall that the Hardy space $H^2(\D)$ is a closed subspace of the Hilbert space $L^2(\T)$. For $g\in L^\infty(\T)$, the operator $T_g\in L(H^2(\D))$, $T_g=PM_g$ is called the {\it Toeplitz operator with the symbol $g$}, where $P$ is the orthogonal projection in $L^2(\T)$ onto $H^2(\D)$ and $M_g$ is the multiplication by $g$ operator $M_gf(z)=g(z)f(z)$. The Toeplitz operator $T_g$ is called {\it analytic} if $g\in H^\infty(\D)$ in which case $T_g$ is the restriction of $M_g$ to $H^2(\D)$. Similarly, $T_g$ is called {\it coanalytic} if $\overline{g}\in H^\infty(\D)$. The coanalytic Toeplitz operators are exactly the adjoints of the analytic Toeplitz operators since $T_g^*=T_{\overline{g}}$. The hypercyclic coanalytic Toeplitz operators were characterized by  Godefroy and Shapiro \cite{gs}, see also \cite{kc}.

\begin{thmgs} Let $g\in H^\infty(\D)$. Then the coanalytic Toeplitz operator $T_g^*\in L(H^2(\D))$ is norm hypercyclic if and only if $g$ is non-constant and $g(\D)\cap \T\neq\varnothing$.
\end{thmgs}

The question of characterizing $1$-weakly hypercyclic coanalytic Toeplitz operators is raised in \cite{fe}. Note that if $g\in H^\infty(\D)$ and $g(\D)\cap \D\neq\varnothing$, then either $T^*_g$ is norm hypercyclic according to Theorem~GS or $T^*_g$ is a contraction and can not be 1-weakly hypercyclic. Thus the only case to consider is when $g(\D)\cap \D=\varnothing$. The following two results provide a partial answer to the above question.

\begin{theorem}\label{toe01} Let $g\in H^\infty(\D)$ be non-constant and such that $g(\D)\cap \D=\varnothing$ and $\{z\in\T:|g(z)|=1\}$ is a subset of $\T$ of positive Lebesgue measure. Then the
coanalytic Toeplitz operator $T^*_g$ is weakly hypercyclic.
\end{theorem}

Curiously enough, from a result \cite[Theorem~2.8]{berg} of Bourdon and Shapiro it follows that if $g$ satisfies the conditions of Theorem~\ref{toe01} and $T^*_g$ extends to a continuous linear operator on the Bergman space $A^2(\D)$, then actually (the extended) $T^*_g$ is norm hypercyclic on $A^2(\D)$.

\begin{theorem}\label{toe02} Let $g\in H^\infty(\D)$ be such that $g(\D)\cap \D=\varnothing$, $|g|>1$ almost everywhere on $\T$ and
$\log(|g|-1)\in L^1(\T)$. Then $T^*_g$ is not $1$-weakly hypercyclic.
\end{theorem}

\begin{corollary}\label{toe03} Let $g\in H^\infty(\D)$  be such that $g(\D)\cap\D=\varnothing$ and $g$ extends analytically to $r\D$ for some $r>1$. Then $T^*_g$ is not $1$-weakly hypercyclic.
\end{corollary}

\begin{proof} Since the case of constant $g$ is trivial, we may assume that $g$ is non-constant. Clearly, $|g|^2-1$ is a non-negative real-analytic function on $\T$. Since $g$ is non-constant and $g(\D)\cap\D=\varnothing$, $|g|$ can not be identically $1$ on $\T$ (use the maximum modulus principle). Thus $|g|^2-1$ is positive on $\T$ with the exception of finitely many zeros each of which has finite order. Hence $\log(|g|^2-1)$ can have only finitely many singularities each of which is logarithmic. It follows that $\log(|g|^2-1)\in L^1(\T)$ and therefore $\log(|g|-1)=\log(|g|^2-1)-\log(|g|+1)\in L^1(\T)$. It remains to apply Theorem~\ref{toe02}.
\end{proof}

Of course, the analyticity of $g$ on $\T$ condition in the above corollary can be replaced by the condition of $g$ on $\T$ being from a quasi-analyticity class. The proof remains virtually the same.

Feldman \cite{fe} has asked whether a $1$-weakly hypercyclic operator on a Hilbert space must be norm supercyclic. The following result provides a negative answer to this question.

\begin{proposition}\label{nns} There is a $1$-weakly hypercyclic operator $T\in L(\ell_2)$ such that $T^2$ is non-cyclic. In particular, $T$ is not $2$-weakly supercyclic.
\end{proposition}

The last assertion is due to Feldman \cite{fe}, who had observed that $T^n$ must be cyclic if $T$ is an $n$-weakly supercyclic operator on a Banach space. Finally, we answer affirmatively the question \cite{fe} whether $n$-weak hypercyclicity for every $n\in\N$ implies weak hypercyclicity.

\begin{proposition}\label{ninf} Let $X$ be a separable infinite dimensional Banach space and $T\in L(X)$. Then $T$ is weakly hypercyclic if and only if $T$ is $n$-weakly hypercyclic for every $n\in\N$. Similarly, $T$ is weakly supercyclic if and only if $T$ is $n$-weakly supercyclic for every $n\in\N$.
\end{proposition}

Our way of proving Proposition~\ref{ninf} has the following surprising byproduct. It is well-known and easy to see that the set of hypercyclic vectors of every norm hypercyclic operator on a Banach space $X$ is a dense $G_\delta$-set. The canonical proof relies heavily upon the second countability of $X$. It was widely believed that the same statement does not hold for the set of weakly hypercyclic vectors. Curiously enough it does.

\begin{proposition}\label{ninf1} Let $X$ be a separable infinite dimensional Banach space and $T\in L(X)$. Then the set of weakly hypercyclic vectors for $T$ either is empty or is a dense $G_\delta$ subset of $X$ $($both the density and the $G_\delta$-property are with respect to the norm topology$)$. The same dichotomy holds for the set of weakly supercyclic vectors of $T$.
\end{proposition}

We derive Theorem~\ref{toe01} from Theorem~\ref{qq} and derive Theorem~\ref{toe02} from Theorem~\ref{norm1} in Section~\ref{s2}. We prove Theorems~\ref{norm1} and derive few corollaries in Section~\ref{s3}. Theorems~\ref{qq} and~\ref{qq1} together with some more general statements are proved in Section~\ref{s4}. Section~\ref{s5} is devoted to the proof of Propositions~\ref{nns}, \ref{ninf} and~\ref{ninf1} as well as to concluding remarks and open questions. It is worth noting that we shall actually prove Propositions~\ref{ninf} and~\ref{ninf1} in the more general setting of Fr\'echet space operators.

While proving Theorem~\ref{toe02}, we shall see that $\|(T^*_g)^nf\|\geq cn$ for some $c=c(f,g)>0$ whenever $f$ is a non-zero element of $H^2(\D)$ and $g$ satisfies the conditions of Theorem~\ref{toe02}. In fact, we can prove a number of related results on the orbits of expanding coanalytic Toeplitz operators. These estimates are collected in the Appendix.

\section{Weak hypercyclicity of coanalytic Toeplitz operators \label{s2}}

\begin{lemma}\label{to0}Let $g\in H^\infty(\D)$ be such that $g(\D)\subseteq \D$. Then $\|(T_g')^n\psi\|\to 0$ and  $\|(T_g^*)^nf\|\to0$ for every $f\in H^2(\D)$ and $\psi\in H^2(\D)^*$.
\end{lemma}

\begin{proof} Since $g(\D)\subseteq \D$, $\|T_g'\|=\|T_g\|=\|g\|_{H^\infty}\leq 1$. In particular, $T_g'$ is power bounded. For each $w\in \D$, the evaluation map $\delta_w(f)=f(w)$ is a continuous linear functional on $H^2(\D)$. It is easy to see that $(T'_g)^n\delta_w=g(w)^{n}\delta_w$. Since $g(\D)\subseteq \D$, $|g(w)|<1$ and therefore $\|(T_g')^{n}\delta_w\|\to 0$. Hence $\|(T_g')^{n}\psi\|\to 0$ for every $\psi\in E=\spann\{\delta_w:w\in \D\}$. Since the functionals $\delta_w$ for $w\in \D$ separate the points of the reflexive Banach space $H^2(\D)$, $E$ is dense in $H^2(\D)^*$. By Lemma~\ref{pb}, $\|(T_g')^{n}\psi\|\to 0$ for every $\psi\in H^2(\D)^*$.
By Remark~\ref{rire}, $T_g^*$ and $T_g'$ are $\R$-linearly similar and therefore $\|(T_g^*)^nf\|\to0$ for every $f\in H^2(\D)$.
\end{proof}

First, we prove Theorem~\ref{toe01} taking Theorem~\ref{qq} as granted.

\begin{proof}[Proof of Theorem~$\ref{toe01}$]Since $|g(z)|>1$ for $z\in \D$, $h=\frac1{g}$ belongs to $H^\infty(\D)$ and satisfies $h(\D)\subseteq\D$.
Note that $T_g$ is invertible and $T_g^{-1}=T_h$. By Lemma~\ref{to0}, $\|(T'_g)^{-n}\psi\|=\|(T'_h)^n\psi\|\to 0$ for every $\psi\in H^2(\D)^*$. That is, every $\psi\in H^2(\D)^*$ has a norm convergent to 0 backward $T_g'$-orbit.

Since $H^2(\D)$ is a closed subspace of $L^2(\T)$, for every $\phi\in L^2(A,\lambda)$, $G_\phi(f)=\int_A f(z)\phi(z)\,d\lambda(z)$ is a continuous linear functional on $H^2(\D)$. Furthermore, the map $\Phi(\phi)=G_\phi$ is an injective continuous linear operator from $L^2(A,\lambda)$ to $H^2(G)^*$. The injectivity follows from the fact that a non-zero $f\in H^2(\D)$ can not vanish on a subset of $\T$ of positive Lebesgue measure.
Moreover, it is easy to check that $T_g'(G_\phi)=G_{S\phi}$, where $S\phi(z)=g(z)\phi(z)$. Thus, $F=\Phi(L^2(A,\lambda))$ is an $T_g'$-invariant subspace of $H^2(\D)^*$, the inner product $\langle\cdot,\cdot\rangle$ on $F$ transferred from $L^2(A,\lambda)$ by the operator $\Phi$ defines the Hilbert space topology stronger than the one inherited from the norm topology of $H^2(\D)^*$ and the restriction of $T$ to $F$ as an operator on the Hilbert space $F$ is unitarily equivalent to $S$. Now since $|g|=1$ on $A$, $S$ is unitary. Since an $H^\infty$ function can not be constant on a subset of $\T$ of positive Lebesgue measure, $g$ can not take one value on a subset of $A$ of positive Lebesgue measure. It follows that the point spectrum of $S$ is empty and therefore $S$ has no finite dimensional invariant subspaces. Hence the restriction of $T$ to $F$ as an operator on the Hilbert space $F$ (being unitarily equivalent to $S$) is an isometry with no non-trivial invariant subspaces. Thus conditions (B1--B3) for $T=T_g'$ are satisfied. By Theorem~\ref{qq}, $T_g'$ is weakly hypercyclic. Since $T_g^*$ is $\R$-linearly similar to $T_g'$, $T_g^*$ is weakly hypercyclic as well.
\end{proof}

\begin{lemma}\label{gege} Let $\hh$ be a Hilbert space and $A,B\in L(\hh)$ be such that $AB=BA$ and $B$ is invertible. Then $A^*A\leq B^*B$ if and only if $AA^*\leq BB^*$. Furthermore, $A^*A<B^*B$ if and only if $AA^*<BB^*$.
\end{lemma}

\begin{proof} Since $B$ is invertible, we can consider the operator $T=AB^{-1}$. Using the equality $AB=BA$, one can easily verify that $A^*A\leq B^*B\iff T^*T\leq I$, $A^*A< B^*B\iff T^*T< I$, $AA^*\leq BB^*\iff TT^*\leq I$ and $AA^*< BB^*\iff TT^*< I$. Now it is straightforward to see that each of the inequalities $T^*T\leq I$ and $TT^*\leq I$ is equivalent to $\|T\|\leq 1$, while each of the inequalities $T^*T<I$ and $TT^*< I$ is equivalent to $|\langle Tx,y\rangle|<1$ for every $x,y\in\hh$ satisfying  $\|x\|\leq 1$ and $\|y\|\leq 1$.
\end{proof}

\begin{remark}\label{rer} On can not scrape the invertibility condition in the above lemma. It can be somewhat relaxed however. By means of applying the Douglas lemma, one can see that the invertibility of $B$ can be replaced by the assumption that $B$ is injective and has dense range.
\end{remark}

\begin{lemma} \label{sumtoe} Let $g_1,\dots,g_n,h_1,\dots,h_m\in H^\infty(\D)$, $S=\sum\limits_{j=1}^m T_{h_j}^*T_{h_j}-\sum\limits_{k=1}^n T_{g_k}^*T_{g_k}$ and $H\in L^\infty(\T)$ be defined by $H=\sum\limits_{j=1}^m |h_j|^2-\sum\limits_{k=1}^n |g_k|^2$.
Then $S\geq 0$ if and only if $H\geq 0$ almost everywhere on $\T$. Furthermore,
$S>0$ if and only if $H\geq 0$ almost everywhere on $\T$ and $H\neq 0$ $(=$is positive on a set of positive measure$).$
\end{lemma}

\begin{proof}It is straightforward to see that $\langle Sf,f\rangle=\int_\T |f|^2 H\,d\lambda$ for every $f\in H^2(\D)$. Note that the inner/outer factorization theorem for Hardy space functions easily yields that the set $\{|f|^2:f\in H^2(\D)\}$ is dense in $L^1_+(\T)=\{h\in L^1(\T):h\geq 0\}$ with respect to the metric inherited from $L^1(\T)$. Thus the condition $H\geq 0$ is equivalent to $S\geq 0$. Similar considerations together with the fact that a non-zero $f\in H^2(\D)$ is almost everywhere on $\T$ different from zero imply that $S>0$ if and only if $H\geq 0$ and $H>0$ on a subset of $\T$ of positive Lebesgue measure.
\end{proof}

Recall that $H^\infty(\D)$ carries the natural structure of a Banach algebra. We say that $g\in H^\infty(\D)$ is {\it invertible} if $g$ is invertible as an element of $H^\infty(\D)$. Of course, $g$ is invertible if and only if there is $\epsilon>0$ such that $|g|>\epsilon$ everywhere on $\D$.

\begin{corollary}\label{twotoe} Let $h,g\in H^\infty(\D)$ and $g$ be invertible. Then $T_hT^*_h\leq T_gT^*_g$ if and only if $|h|\leq |g|$ almost everywhere on $\T$. Furthermore, $T_hT^*_h<T_gT^*_g$ if and only if $|h|\leq |g|$ almost everywhere on $\T$ and $|h|<|g|$ on a subset of $\T$ of positive Lebesgue measure.
\end{corollary}

\begin{proof} By Lemma~\ref{sumtoe}, $T_h^*T_h\leq T_g^*T_g$ if and only if $|h|\leq |g|$ almost everywhere on $\T$ and  $T^*_hT_h<T^*_gT_g$ if and only if $|h|\leq |g|$ almost everywhere on $\T$ and $|f|<|g|$ on a subset of $\T$ of positive Lebesgue measure. Obviously $T_h$ and $T_g$ commute and $T_g$ is invertible with $T_g^{-1}=T_{1/g}$. By Lemma~\ref{gege}, $T_hT_h^*\leq T_gT_g^*$
if and only if $T_h^*T_h\leq T_g^*T_g$ and $T_hT_h^*< T_gT_g^*$
if and only if $T_h^*T_h< T_g^*T_g$. The required result immediately follows.
\end{proof}

\begin{lemma}\label{mto1} Let $h_1,\dots,h_m,g\in H^\infty(\D)$ be such that $g$ is invertible and $\sum\limits_{j=1}^m|h_j|^2\leq |g|^2$ almost everywhere on $\T$. Then $\sum\limits_{k=1}^n T_{h_k}T_{h_k}^*\leq T_{g}T_{g}^*$.
\end{lemma}

\begin{proof} Denote $f_j=\frac{h_j}{g}$ for $1\leq j\leq m$. It is easy to see that the required inequality $\sum\limits_{k=1}^n T_{h_k}T_{h_k}^*\leq T_{g}T_{g}^*$ is equivalent to $\sum\limits_{k=1}^n T_{f_k}T_{f_k}^*\leq I$. From the definition it immediately follows that every analytic Toeplitz operator is hyponormal (see \cite{cow} for a complete characterization of hyponormal Toeplitz operators). Since $\sum\limits_{j=1}^m|f_j|^2\leq 1$ almost everywhere on $\T$, Lemma~\ref{sumtoe} yields $\sum\limits_{k=1}^n T_{f_k}^*T_{f_k}\leq I$. By the hyponormality, $T_{f_k}T_{f_k}^*\leq T_{f_k}^*T_{f_k}$. Hence $\sum\limits_{k=1}^n T_{f_k}T_{f_k}^*\leq \sum\limits_{k=1}^n T_{f_k}^*T_{f_k}\leq I$ and the required inequality follows.
\end{proof}

Recall that $h\in H^1(\D)$ is called an {\it outer function} if there is a real valued function $q\in L^1(\T)$ such that $e^q\in L^1(\T)$ and $h:\D\to\C$ is given by the formula
\begin{equation}\label{outer}
h(z)=\exp\biggl(-\int_\T \frac{\xi+z}{\xi-z}q(\xi)\,d\lambda(\xi)\biggr).
\end{equation}
Moreover, the function in the right-hand side always belongs to $H^1(\D)$ and belongs to $H^\infty(\D)$ if and only if $q$ is bounded above. Furthermore, $\log|h|=q$ almost everywhere on $\T$.

\begin{lemma}\label{new} Let $g\in H^\infty(\D)$ be such that $g(\D)\cap\D=\varnothing$, $|g|>1$ almost everywhere on $\T$ and $\log(|g|-1)\in L^1(\T)$. Then there is an outer function $h\in H^\infty(\D)$ such that $|h|\leq |g|-1$ almost everywhere on $\T$.
\end{lemma}

\begin{proof} Since $\log(|g|-1)\in L^1(\T)$ is real valued and bounded above, the formula (\ref{outer}) with $q=\log(|g|-1)$ defines an outer function $h\in H^\infty(\D)$ such that $\log|h|=\log(|g|-1)$. Hence $|h|=|g|-1$ almost everywhere on $\T$ and $h$ satisfies all desired conditions.
\end{proof}

Now we shall prove Theorem~\ref{toe02}, taking Theorem~\ref{norm1} as granted.

\begin{proof}[Proof of Theorem~$\ref{toe02}$] By Lemma~\ref{new}, there is a non-zero $h\in H^\infty(\T)$ such that $|h|\leq |g|-1$ almost everywhere on $\T$. Then $|h|^2\leq (|g|-1)^2\leq (|g|-1)(|g|+1)=|g|^2-1$ almost everywhere on $\T$. By Lemma~\ref{mto1},
$T_hT_h^*+I\leq T_gT_g^*$. Since $T_h^*$ and $T_g^*$ commute and $(T_h^*)^2=T_{h^2}^*\neq 0$, Theorem~\ref{norm1} implies that $T^*_g$ is not 1-weakly hypercyclic.
\end{proof}

\section{Theorem~\ref{norm1} and its applications\label{s3}}

\begin{lemma}\label{qqq1}
Let $\hh$ be a complex Hilbert space and $T,S\in L(\hh)$ be such that $S^2x\neq 0$, $TS=ST$ and $T^*T\geq S^*S+I$.
Then $\|T^nx\|^2\geq \frac{n(n-1)}{2}\|S^2x\|^2$ for every $x\in\hh$ and $n\in\N$.
\end{lemma}

\begin{proof}Since $T^*T\geq S^*S+I$, we have $\langle (T^*T-I-S^*S)x,x\rangle\geq 0$ for every $x\in \hh$.
That is, $\|Tx\|^2\geq \|x\|^2+\|Sx\|^2$ for every $x\in \hh$.
Since $S$ and $T$ commute, applying the last inequality with $T^nx$ in the place of $x$, we obtain
\begin{equation}\label{e1}
\|T^{n+1}x\|^2\geq \|T^{n}x\|^2+\|T^n(Sx)\|^2\ \ \text{for every $x\in \hh$ and every $n\in\Z_+$}.
\end{equation}
Since $T^*T\geq I$, $\|Tx\|\geq \|x\|$ for every $x\in\hh$. Thus (\ref{e1}) implies
$$
\|T^{n+1}x\|^2\geq \|T^{n}x\|^2+\|Sx\|^2\ \ \text{for every $x\in \hh$ and every $n\in\Z_+$}.
$$
Hence $\|T^nx\|^2\geq n\|Sx\|^2$ for every $x\in\hh$ and $n\in\N$.
Applying this inequality with $Sx$ plugged in instead of $x$, we see that
$\|T^n(Sx)\|^2\geq n\|S^2x\|^2$ for every $x\in\hh$ and $n\in\Z_+$.
Using this inequality together with (\ref{e1}), we get
$$
\|T^{n+1}x\|^2\geq \|T^{n}x\|^2+n\|S^2x\|^2\ \ \text{for every $x\in\hh$ and $n\in\Z_+$}.
$$
Summing up the first $n$ of the above inequalities, we obtain
$$
\|T^nx\|^2\geq \|x\|^2+\|S^2x\|^2(1+2+{\dots}+(n-1))\geq\|S^2x\|^2\frac{n(n-1)}{2}\ \ \text{for every $x\in\hh$ and $n\in\N$},
$$
as required.
\end{proof}

We shall also need the following beautiful result of Ball \cite{ball,ball1}.

\begin{thmb} Let $X$ be a real or complex Banach space and $\{x_n\}_{n\in\N}$ be a sequence in $X$.

If $\sum\limits_{n=1}^\infty \|x_n\|^{-1}<1$, then there is $f\in X^*$ such that $\|f\|\leq 1$ and $|f(x_n)|\geq 1$ for each $n\in\N$.

If $X$ is a complex Hilbert space and $\sum\limits_{n=1}^\infty \|x_n\|^{-2}\leq 1$, then there is $y\in X$ such that $\|y\|\leq 1$ and $|\langle x_n,y\rangle|\geq 1$ for each $n\in\N$.
\end{thmb}

\begin{proof}[Proof of Theorem~$\ref{norm1}$]
Take an arbitrary $x\in\hh$. First, consider the case $S^2x=0$. Since $T$ and $S$ commute the closed linear space $Y=\ker S^2$ is $T$-invariant. Since $S^2\neq 0$, $Y\neq\hh$. Thus, there is a non-zero $y\in\hh$, which is orthogonal to $Y$. It follows that $\langle T^nx,y\rangle=0$ for each $n\in\Z_+$ and therefore $x$ is not a 1-weakly hypercyclic vector for $T$ (actually, it is not even a cyclic vector).

It remains to consider the case $S^2x\neq 0$. In this case Lemma~\ref{qqq1} ensures that $\sum\limits_{n=0}^\infty \|T^nx\|^{-2}<\infty$. By Theorem~B, there is $y\in\hh$ such that $|\langle T^nx,y\rangle|\geq 1$ for each $n\in\Z_+$. Hence $x$ is not a 1-weakly hypercyclic vector for $T$ and therefore $T$ is not 1-weakly hypercyclic.
\end{proof}

We derive a curious corollary from Theorem~\ref{norm1}.

\begin{corollary}\label{coco}Let $\hh$ be a complex Hilbert space and $S\in L(\hh)$, $\|S\|\leq 1$. Then for every $c>0$, the operator $T=(c+1)I+cS$ is
not $1$-weakly hypercyclic.
\end{corollary}

\begin{proof}
Consider the operator $R=\sqrt{c(c+1)}(I+S)$. It is straightforward to verify that
$$
T^*T-R^*R-I=((c+1)I+S^*)((c+1)I+S)-c(c+1)(I+S^*)(I+S)-I=c(I-S^*S).
$$
Since $\|S\|\leq 1$, we have $I-S^*S\geq 0$. Hence $T^*T-R^*R-I\geq 0$. By Theorem~\ref{norm1}, $T$ is not 1-weakly hypercyclic unless $R^2=0$. In the latter case $T$ satisfies a degree two polynomial identity and therefore is not 1-weakly hypercyclic as well.
\end{proof}

Feldman \cite{fe} has conjectured that the operator $2I+B^*$ is not 1-weakly hypercyclic, where $B^*$ is the backward shift on the complex Hilbert space $\ell_2(\Z_+)$. Obviously, Corollary~3.2 proves this conjecture. For what it matters, the same also follows from Corollary~\ref{toe03} since $2I+B^*$ is unitarily equivalent to the coanalytic Toeplitz operator $T^*_g$ with $g(z)=2+z$.

\section{Proof of Theorems~\ref{qq} and~\ref{qq1}\label{s4}}

We need a lot of preparation.
The following lemma follows from Lemmas~2.3 and 2.6 in \cite{ss}. For the sake of convenience, we provide an independent proof.

 \begin{lemma}\label{le0}
Let $\{a_n\}_{n\in\N}$ be a sequence in a pre-Hilbert space $\hh$ such that $\sum\limits_{n=1}^\infty \|a_n\|^{-2}=\infty$ and\hfill\break $\sum\limits_{1\leq m<n<\infty}\frac{|\langle a_m,a_n\rangle|^2}{\|a_n\|^2\|a_m\|^2}<\infty$. Then $0$ is in the closure of the set $\{a_n:n\in\N\}$ in the weak topology.
\end{lemma}

\begin{proof}Passing to the completion of $\hh$ does not change a thing. Thus without loss of generality $\hh$ is a Hilbert space. Obviously, none of the vectors $a_n$ is zero. Denote $c_n=\|a_n\|$ and $v_n=\frac{a_n}{\|a_n\|}$. Then $\|v_n\|=1$, $c_n>0$ and $a_n=c_nv_n$ for each $n\in\N$. The conditions imposed upon $a_n$ now read:
\begin{equation}
\sum_{n=1}^\infty c_n^{-2}=\infty\ \ \text{and}\ \ r=\!\!\!\!\sum\limits_{1\leq m<n<\infty}|\langle v_m,v_n\rangle|^2<\infty.\label{w01}
\end{equation}
First, we shall demonstrate that there is a unique continuous linear operator $T:\ell_2(\N)\to \hh$ such that $Te_n=v_n$ for every $n\in\N$, where $\{e_n\}_{n\in\N}$ is the standard orthonormal basis in $\ell_2(\N)$. This is equivalent to showing that there is a positive constant $d$ such that
\begin{equation}
\|z_1v_1+{\dots}+z_mv_m\|^2\leq d(|z_1|^2+{\dots}+|z_m|^2)\ \ \text{for every $m\in\N$ and $z_1,\dots,z_m\in\K$.}\label{es}
\end{equation}
Let $m\in\N$ and $z_1,\dots,z_m\in\K$. Since
$\|z_1v_1+{\dots}+z_mv_m\|^2=\sum\limits_{1\leq j,k\leq m}z_j\overline{z_k}\langle v_j,v_k\rangle$ and
$\langle v_j,v_j\rangle=1$ for every $j$,
\begin{equation}\label{uu}
\|z_1v_1+{\dots}+z_mv_m\|^2\leq \sum_{j=1}^m |z_j|^2+2\sum_{1\leq j<k\leq m}|z_j||z_k||\langle v_j,v_k\rangle|.
\end{equation}
Applying the Cauchy--Schwartz inequality to the last sum, we get
$$
\sum_{1\leq j<k\leq m}|z_j||z_k||\langle v_j,v_k\rangle|\leq \biggl(\sum_{1\leq j<k\leq m}|z_j|^2|z_k|^2\biggr)^{1/2}\biggl(\sum_{1\leq j<k\leq m}|\langle v_j,v_k\rangle|\biggr)^{1/2}
\leq\sqrt{r}\biggl(\sum_{1\leq j<k\leq m}|z_j|^2|z_k|^2\biggr)^{1/2},
$$
where $r$ is defined in (\ref{w01}). Obviously,
$$
\sum_{1\leq j<k\leq m}|z_j|^2|z_k|^2\leq \frac12\sum_{1\leq j,k\leq m}|z_j|^2|z_k|^2=\frac12\biggl(\sum_{j=1}^m |z_j|^2\biggr)^2.
$$
By the last two displays $
\sum\limits_{1\leq j<k\leq m}|z_j||z_k||\langle v_j,v_k\rangle|\leq \sqrt{\frac r2}\sum\limits_{j=1}^m |z_j|^2$.
Plugging this estimate into (\ref{uu}), we get
$\|z_1v_1+{\dots}+z_mv_m\|^2\leq (1+\sqrt{r/2})\sum\limits_{j=1}^m |z_j|^2$,
which proves (\ref{es}) with $d=1+\sqrt{r/2}$. Thus there is a continuous linear operator $T:\ell_2(\N)\to\hh$ satisfying $Te_n=v_n$ for every $n\in\N$. Obviously, $T(c_ne_n)=a_n$ for every $n\in\N$.
Since $T$ is continuous, in order to show that $0$ is in the closure of the set $A=\{a_n:n\in\N\}$ in the weak topology, it suffices to verify that $0$ is in the closure of the subset $B=\{c_ne_n:n\in\N\}$ of $\ell_2(\N)$ in the weak topology. Indeed, this implication holds since $T(B)=A$. Assume the contrary. That is, $0$ is not in the closure of $B$ in the weak topology. Then there exist $y_1,\dots,y_m\in\ell_2(\N)$ such that $\sum\limits_{j=1}^m|\langle y_j,c_ne_n\rangle|^2\geq 1$ for every $n\in\N$. Hence,
$\sum\limits_{j=1}^m|\langle y_j,e_n\rangle|^2\geq c_n^{-2}$ for every $n\in\N$.
Summing these inequalities up, we obtain
$\sum\limits_{j=1}^m\|y_j\|^2=\sum\limits_{j=1}^m\sum\limits_{n=1}^\infty|\langle y_j,e_n\rangle|^2\geq \sum\limits_{n=1}^\infty c_n^{-2}$,
which contradicts (\ref{w01}) and thus completes the proof.
\end{proof}

\begin{corollary}\label{co0}Let $\{a_n\}_{n\in\N}$ be a sequence in a pre-Hilbert space $\hh$ such that $\sum\limits_{n=1}^\infty \|a_n\|^{-2}=\infty$ and $\sum\limits_{1\leq m<n<\infty}|\langle a_m,a_n\rangle|^2<\infty$. Then $0$ is in the closure of the set $\{a_n:n\in\N\}$ in the weak topology.
\end{corollary}

\begin{proof} If $0$ is in the closure of the set $\{\|a_n\|:n\in\N\}$ of positive numbers, then $0$ is in the norm closure of the set $\{a_n:n\in\N\}$ and therefore belongs to its closure in the weak topology. It remains to consider the case when $0$ is not in the closure of  $\{\|a_n\|:n\in\N\}$. Then there is $c>0$ such that $\|a_n\|\geq c$ for each $n\in\N$. Hence, the condition $\sum\limits_{1\leq m<n<\infty}|\langle a_m,a_n\rangle|^2<\infty$ implies $\sum\limits_{1\leq m<n<\infty}\frac{|\langle a_m,a_n\rangle|^2}{\|a_n\|^2\|a_m\|^2}<\infty$. By Lemma~\ref{le0}, $0$ is in the closure of the set $\{a_n:n\in\N\}$ in the weak topology.
\end{proof}

\begin{lemma}\label{le1} Let $\{c_n\}_{n\in\N}$ be an arbitrary sequence of positive numbers and $\{d_n\}_{n\in\N}$ be such a sequence of positive numbers that $\lim\limits_{n\to\infty}\frac{n}{d_n}=0$. Then there exists a map $\phi:\N\to\N$ such that
\begin{itemize}\itemsep=-2pt
\item[\rm(\ref{le1}.1)] $A_n=\phi^{-1}(n)$ is an infinite subset of $\N$ for each $n\in\N;$
\item[\rm(\ref{le1}.2)] $\lim\limits_{m\to\infty}\frac1{d_m}\sum\limits_{j=1}^{k_{n,m}}c_{\phi(j)}=0$ for every
$n\in\N$, where $\{k_{n,m}\}_{m\in\N}$ is the strictly increasing sequence of positive integers such that $A_n=\{k_{n,m}:m\in\N\}.$
\end{itemize}
\end{lemma}

\begin{proof} Denote $v_n=\max\{c_1,\dots,c_n\}$ and $w_n=\frac{d_n}{n}$. Obviously, the sequence $\{v_n\}$ increases and $w_n\to\infty$. Pick a strictly increasing sequence $\{r_j\}_{j\in\N}$ of positive integers such that
\begin{align}
&\lim\limits_{s\to\infty}\frac{sv_s}{w_{r_{s-1}}}=0,\label{ee1}
\\
&\lim\limits_{s\to\infty}\frac{r_{s-1}}{r_{s}}=0.\label{ee2}
\end{align}
We define the required map $\phi:\N\to\N$ according to the rule:
$$
(\phi(1),\phi(2),\dots)=(\underbrace{1,1,\dots,1}_{\hbox{$r_1$}\ \rm times},\underbrace{1,2,1,2,\dots,1,2}_{\hbox{$r_2$}\ \rm times},
\underbrace{1,2,3,1,2,3,\dots,1,2,3}_{\hbox{$r_3$}\ \rm times},\dots).
$$
Condition (\ref{le1}.1) is obviously satisfied. It remains to verify (\ref{le1}.2). Pick $n\in\N$. For any sufficiently large $m\in\N$, there is a unique integer $s=s(m)\geq n+1$ and a unique integer $l=l(m)$ satisfying $1\leq l\leq r_s$ such that $m=r_n+{\dots}+r_{s-1}+l$. It is easy to see that $k_{n,m}\leq r_1+{\dots}+(s-1)r_{s-1}+ls$ and that $c_{\phi(j)}\leq v_s$ for $j\leq k_{n,m}$. Therefore
$$
\frac1{d_m}\sum_{j=1}^{k_{n,m}}c_{\phi(j)}\leq \frac{v_sk_{n,m}}{mw_m}\leq\frac{v_s(r_1+{\dots}+(s-1)r_{s-1}+ls)}
{w_m(r_n+{\dots}+r_{s-1}+l)}=\frac{sv_s}{w_m}\frac{\frac{r_1}{s}+\frac{2r_2}{s}+{\dots}+\frac{(s-1)r_{s-1}}{s}+l}
{r_n+{\dots}+r_{s-1}+l}.
$$
The equality~(\ref{ee2}) implies that
$$
\lim_{m\to\infty}\frac{\frac{r_1}{s}+\frac{2r_2}{s}+{\dots}+\frac{(s-1)r_{s-1}}{s}+l}
{r_n+{\dots}+r_{s-1}+l}=1.
$$
Since $m>r_{s-1}$, (\ref{ee1}) ensures that
$$
\lim_{m\to\infty}\frac{sv_s}{w_m}=0.
$$
The last three displays combined yield $\lim\limits_{m\to\infty}\frac1{d_m}\sum\limits_{j=1}^{k_{n,m}}c_{\phi(j)}=0$,
as required.
\end{proof}

\subsection{Two general statements}

\begin{theorem}\label{tt0}
Let $X$ be a Banach space and $T\in L(X)$. Assume that there exists a double sequence $\{u_{k,n}\}_{k\in\N,\,n\in\Z}$
of elements of $X$ and an infinite subset $A$ of $\N$ such that
\begin{itemize}\itemsep=-2pt
\item[\rm(A1)] $\{u_{k,0}:k\in\N\}$ is dense in $X;$
\item[\rm(A2)] $Tu_{k,n}=u_{k,n+1}$ for every $k\in\N$ and $n\in\Z;$
\item[\rm(A3)] the $T$-invariant space $E=\spann\{u_{k,n}:n\in\Z_+,\,k\in\N\}$ carries an inner product
$\langle\cdot,\cdot\rangle$ such that the corresponding norm $\|x\|_0=\sqrt{\langle x,x\rangle}$ on $E$ defines a topology
$($not necessarily strictly$)$ stronger than the one inherited from $X;$
\item[\rm(A4)] for every $k\in\N$, $\lim\limits_{n\to\infty\atop n\in A}\|u_{k,-n}\|=0;$
\item[\rm(A5)] for every $k\in\N$, the sequence $\{\|u_{k,n}\|_0\}_{n\in\Z_+}$ is bounded and for
every $k,a\in\N$, $b\in\Z_+$ and $s\in A$, $\lim\limits_{n\to\infty\atop n\in A}\langle u_{k,n-s},u_{a,b}\rangle=0;$
\item[\rm(A6)] for every $k,m\in\N$ and $s\in A$, $\lim\limits_{t\to\infty\atop t\in A,\,t>s}\slim\limits_{n\to\infty\atop n\in A,\,n>t}|\langle u_{k,n-s},u_{m,n-t}\rangle|=0$.
\end{itemize}
Then the operator $T$ is weakly hypercyclic.
\end{theorem}

\begin{theorem}\label{tt}
Let $X$ be a Banach space and $T\in L(X)$. Assume that there exists a double sequence $\{u_{n,k}\}_{k\in\N,\,n\in\Z}$
of elements of $X$ such that conditions {\rm(A1--A3)} are satisfied and
\begin{itemize}\itemsep=-2pt
\item[\rm(A$4'$)] for every $k\in\N$, $\ilim\limits_{n\to\infty}\|u_{k,-n}\|=0;$
\item[\rm(A$5'$)] for every $k,a\in\N$ and $b\in\Z_+$, $\lim\limits_{n\to\infty}\langle u_{k,n},u_{a,b}\rangle=0;$
\item[\rm(A$6'$)] for every $k,m\in\N$, $\lim\limits_{j\to+\infty}\sup\limits_{n\in \Z_+}|\langle
u_{k,n+j},u_{m,n}\rangle|=0.$
\end{itemize}
Then the operator $T$ is weakly hypercyclic.
\end{theorem}

We shall prove these two results in one go.

\begin{proof}[Proof of Theorems~$\ref{tt0}$ and~$\ref{tt}$] We suppose that the assumptions of either Theorem~\ref{tt0} or of Theorem~\ref{tt} are satisfied.
Each of the conditions (A5) and (A$5'$) ensures that for every $k\in\N$,
$$
c_k=\sup\{\|u_{k,n}\|_0:n\in\Z_+\}<\infty.
$$
Thus $\{c_k\}$ is a sequence of positive numbers. By Lemma~\ref{le1}, there is a map $\phi:\N\to\N$ such that
\begin{align}
&\text{$A_k=\phi^{-1}(k)$ is an infinite subset of $\N$ for each $k\in\N;$}\label{e3}
\\
&\begin{array}{l}\text{$\lim\limits_{m\to\infty}\frac1{m\log(m+1)}\sum\limits_{j=1}^{r_{k,m}}c^2_{\phi(j)}=0$ for every
$k\in\N$, where $\{r_{k,m}\}_{m\in\N}$ is the strictly}\\ \text{increasing sequence of positive integers such that $A_k=\{r_{k,m}:m\in\N\}.$}\end{array}\label{e4}
\end{align}

Under the assumptions of Theorem~\ref{tt} we assume $A=\N$. Using (A4--A6) under the assumptions of Theorem~\ref{tt0} and using (A$4'$--A$6'$) under the assumptions of Theorem~\ref{tt}, we can construct inductively a strictly increasing sequence $\{\theta(j)\}_{j\in\N}$ of integers with $\theta(1)=0$ and $\theta(j)\in A$ for $j>0$ such that for every $j\in\N$ (actually (A4--A6) and (A$4'$--A$6'$) are specifically designed for this very purpose),
\begin{align}
&|\langle u_{\phi(s),\theta(r)-\theta(s)},u_{\phi(t),\theta(j)-\theta(t)}\rangle|<2^{-j}
\ \ \text{whenever $1\leq s,t,r<j$};\label{e5}
\\
&|\langle u_{\phi(s),r-\theta(s)},u_{\phi(j),r-\theta(j)}\rangle|<c_{\phi(s)}c_{\phi(j)}4^{-j}
\ \ \text{whenever $1\leq s<j$, $r>\theta(j)$};\label{e6}
\\
&\|u_{\phi(j),-\theta(j)}\|<\|T\|^{-\theta(j-1)}2^{-j}.\label{e7}
\end{align}
Condition (A2) together with each of the conditions (A4) or (A$4'$) implies that $\|T\|>1$. Thus (\ref{e7}) ensures that the series $\sum\limits_{k=1}^\infty u_{\phi(k),-\theta(k)}$
is norm-convergent to some $u\in X$:
$$
u=\sum\limits_{k=1}^\infty u_{\phi(k),-\theta(k)}.
$$
The proof will be complete if we verify that $u$ is a weakly hypercyclic vector for $T$.
By (A2),
\begin{equation}\label{ope}
T^{\theta(k)}u=u_{\phi(k),0}+a_k+b_k,\ \ \text{where}\ \ a_k=\sum_{j=1}^{k-1} u_{\phi(j),\theta(k)-\theta(j)}\ \
\text{and}\ \ b_k=\sum_{j=k+1}^\infty T^{\theta(k)}u_{\phi(j),-\theta(j)}.
\end{equation}
Using (\ref{e7}), we obtain
$$
\|b_k\|\leq \sum_{j=k+1}^\infty \|T\|^{\theta(k)}\|u_{\phi(j),-\theta(j)}\|\leq
\sum_{j=k+1}^\infty \|T\|^{\theta(k)}\|T\|^{-\theta(j-1)}2^{-j}\leq
\sum_{j=k+1}^\infty 2^{-j}=2^{-k}.
$$
Recall that $A_k=\phi^{-1}(k)$. By (\ref{ope}) and the above display,
\begin{equation}\label{ope1}
T^{\theta(r)}u-u_{k,0}=a_r+b_r\ \ \text{for each $r\in A_k$, where}\ \ a_r=\sum_{j=1}^{r-1} u_{\phi(j),\theta(r)-\theta(j)}\ \
\text{and}\ \ \|b_r\|\leq 2^{-r}.
\end{equation}

By (A1), the set $\{u_{k,0}:k\in\N\}$ is dense in $X$. Thus in order to show that $u$ is a weakly hypercyclic vector for $T$, it suffices to demonstrate that each $u_{k,0}$ belongs to the closure of $O(T,u)$ in the weak topology. According to (\ref{ope1}), the latter will be achieved if we verify that $0$ is in the closure of the set $\{a_r:r\in A_k\}$ with respect to the weak topology. Since each $a_r$ belongs to $E$ and the norm $\|\cdot\|_0$ defines a stronger topology on $E$ than the one inherited from $X$, it is enough to show that $0$ is in the closure of $\{a_r:r\in A_k\}$ in the weak topology of the pre-Hilbert space $(E,\|\cdot\|_0)$. By Corollary~\ref{co0}, it suffices to show that
\begin{equation}\label{want}
\text{$\sum\limits_{r\in A_k}\|a_r\|^{-2}=\infty$ \ and \ $\sum\limits_{r,q\in A_k\atop r<q}|\langle a_r,a_q\rangle|^2<\infty$.}
\end{equation}

In order to do that, we need to estimate $\langle a_r,a_q\rangle$ for $r,q\in A_k$, $r\leq q$. According to (\ref{ope1}),
$a_r=\sum\limits_{j=1}^{r-1} u_{\phi(j),\theta(r)-\theta(j)}$. Hence,
$$
\|a_r\|_0^2=\langle a_r,a_r\rangle=\sum_{1\leq j,l\leq r-1}\langle u_{\phi(j),\theta(r)-\theta(j)},u_{\phi(l),\theta(r)-\theta(l)}\rangle.
$$
We split the above sum into the summands with $j=l$ and with $j\neq l$. If $j\neq l$, then by (\ref{e6}),
$$
|\langle u_{\phi(j),\theta(r)-\theta(j)},u_{\phi(l),\theta(r)-\theta(l)}\rangle|<c_{\phi(j)}c_{\phi(l)}4^{-\max\{j,l\}}\leq c_{\phi(j)}c_{\phi(l)}2^{-j-l}.
$$
For $j=l$, we have
$$
|\langle u_{\phi(j),\theta(r)-\theta(j)},u_{\phi(j),\theta(r)-\theta(j)}\rangle|=\|u_{\phi(j),\theta(r)-\theta(j)}\|_0^2\leq c^2_{\phi(j)}.
$$
According to the last three displays,
\begin{equation*}
\|a_r\|_0^2\leq \sum_{j=1}^{r-1}c_{\phi(j)}^2+\sum_{1\leq j,l\leq r-1\atop j\neq l}c_{\phi(j)}c_{\phi(l)}2^{-j-l}
\leq \sum_{j=1}^{r-1}c_{\phi(j)}^2+\sum_{1\leq j,l\leq r-1}c_{\phi(j)}c_{\phi(l)}2^{-j-l}
=\sum_{j=1}^{r-1}c_{\phi(j)}^2+\Bigl(\sum_{j=1}^{r-1}2^{-j}c_{\phi(j)}\Bigr)^2.
\end{equation*}
Applying the Cauchy--Schwartz inequality to the last sum, we obtain
$$
\|a_r\|_0^2\leq \sum_{j=1}^{r-1}c_{\phi(j)}^2+\Bigl(\sum_{j=1}^{r-1}c^2_{\phi(j)}\Bigr)\Bigl(\sum_{j=1}^{r-1}2^{-2j}\Bigr)\leq 2\sum_{j=1}^{r-1}c_{\phi(j)}^2.
$$
Using the above display and (\ref{e4}) we see that $\lim\limits_{m\to\infty}\frac{\|a_{r_{k,m}}\|_0^2}{m\log(m+1)}=0$, where $\{r_{k,m}\}_{m\in\N}$ is the strictly increasing sequence of positive integers such that $A_k=\{r_{k,m}:m\in\N\}$. Since $\sum\limits_{m=1}^\infty \frac1{m\log(m+1)}=\infty$, we arrive to $\sum\limits_{r\in A_k}\|a_r\|^{-2}=\infty$. Thus the first part of (\ref{want}) has been verified.

Now let $r,q\in A_k$, $r<q$. Then $a_r=\sum\limits_{j=1}^{r-1} u_{\phi(j),\theta(r)-\theta(j)}$ and $a_q=\sum\limits_{l=1}^{q-1} u_{\phi(l),\theta(q)-\theta(l)}$ and therefore
$$
\langle a_r,a_q\rangle=\sum_{j=1}^{r-1}\sum_{l=1}^{q-1}\langle u_{\phi(j),\theta(r)-\theta(j)},u_{\phi(l),\theta(q)-\theta(l)}\rangle.
$$
According to (\ref{e5}) for $r,q,j,l$ as in the above display,
$|\langle u_{\phi(j),\theta(r)-\theta(j)},u_{\phi(l),\theta(q)-\theta(l)}\rangle|<2^{-q}$. Plugging these estimates into the above display, we get
$$
|\langle a_r,a_q\rangle|\leq q^22^{-q}.
$$
Hence
$$
\sum_{r\in A_k,\ r<q} |\langle a_r,a_q\rangle|^2\leq q(q^22^{-q})^2=q^54^{-q}.
$$
Summing up over $q$, we get
$$
\sum\limits_{r,q\in A_k\atop r<q}|\langle a_r,a_q\rangle|^2\leq \sum_{q\in A_k}q^54^{-q}\leq \sum_{q=1}^\infty q^54^{-q}<\infty.
$$
Thus both parts of (\ref{want}) are satisfied. This completes the proof of Theorems~\ref{tt0} and~\ref{tt}.
\end{proof}

\subsection{Proof of Theorem~\ref{qq1}}

Since $X$ is separable and $F$ is dense in $X$, we can pick a countable set $\{u_{k,0}:k\in\N\}\subset F$, which is norm dense in $X$. By (B3$'$), we can find $\{u_{k,-n}:k,n\in\N\}\subset X$ such that $Tu_{k,-n}=u_{k,1-n}$ for every $k,n\in\N$ and $\ilim\limits_{n\to\infty}\|u_{k,-n}\|=0$. Now we set $u_{k,n}=T^nu_{k,0}$ for $k,n\in\N$. By Theorem~\ref{tt}, the proof will be complete if we show that the double sequence $\{u_{k,n}\}_{k\in\N,\,n\in\Z}$ satisfies conditions (A1--A3) and (A$4'$--A$6'$), where we equip $E=\spann\{u_{k,n}:n\in\Z_+,\,k\in\N\}$ with the inner product inherited from $F$. Obviously, (A1), (A2), (A3) and (A$4'$) are satisfied. Now let $k,a\in \N$ and $b\in\Z_+$. Then by (B$2'$), $\langle u_{k,n},u_{a,b}\rangle=\langle T^nu_{k,0},u_{a,b}\rangle\to 0$ as $n\to\infty$. Thus (A$5'$) is satisfied. Finally, let $k,m\in N$. By (B$2'$), $T$ restricted to $F$ is an $\|\cdot\|_0$-isometry. Hence for every $j,n\in\Z_+$, $\langle u_{k,n+j},u_{m,n}\rangle=\langle T^nu_{k,j},T^nu_{m,0}\rangle=\langle u_{k,j},u_{m,0}\rangle\to 0$ as $j\to\infty$ according to the already verified (A$5'$). Condition (A$6'$) easily follows. Thus (A1--A3) and (A$4'$--A$6'$) are all satisfied and $T$ is weakly hypercyclic by Theorem~\ref{tt}.

The proof of Theorem~\ref{qq} requires some extra preparation.

\subsection{Auxiliary facts}

We need some extra notation. The symbol $\M(\T)$ will stand for the space of Borel $\sigma$-additive complex valued measures on $\T$ equipped with the full variation norm. It is well-known that $\M(\T)$ is a Banach space and that the subset $\M_c(\T)$ of continuous (=purely non-atomic) measures $\mu\in\M(\T)$ is a closed linear subspace of the Banach space $\M(\T)$. Symbol $\M_+(\T)$ stands for the set of $\mu\in\M(\T)$ taking values in the set $\R_+$ of non-negative real numbers. Recall that for $\mu\in\M(\T)$ and $n\in\Z$, the $n^{\rm th}$ Fourier coefficient of $\mu$ is defined by
$$
\widehat{\mu}(n)=\int_\T z^n\,d\mu(z)\in\C.
$$

\begin{lemma}\label{ele} Let $\mu\in\M_c(\T)$. Then
\begin{equation}\label{m0}
\lim_{n\to\infty}\frac1{n+1}\sum_{k=0}^n|\widehat{\mu}(k)|^2=0.
\end{equation}
\end{lemma}

\begin{proof} It is easy to see that the set of measures $\mu\in\M(\T)$ satisfying (\ref{m0}) is a linear subspace of $\M(\T)$. Since $\spann(\M_c(\T)\cap \M_+(\T))=\M_c(\T)$, it is enough to verify (\ref{m0}) for $\mu\in \M_c(\T)\cap \M_+(\T)$. That is, without loss of generality, we may assume that $\mu\in\M_+(\T)$.

In the latter case, the Fubini theorem yields
$$
|\widehat{\mu}(k)|^2=\widehat{\mu}(k)\overline{\widehat{\mu}(k)}=\int_\T z^k\,d\mu(z)\int_\T w^{-k}\,d\mu(w)=
\int_{\T\times\T}\Bigl(\frac zw\Bigr)^k\,d(\mu\times\mu)(z,w).
$$
Summing these equalities up for $0\leq k\leq n$, we obtain
$$
\frac1{n+1}\sum_{k=0}^n|\widehat{\mu}(k)|^2=\int_{\T\times\T}h_n(z,w)\,d(\mu\times\mu)(z,w),\ \ \text{where}\ \
h_n(z,w)=\frac1{n+1}\sum_{k=0}^n \frac{z^k}{w^k}=\frac{1-(z/w)^{n+1}}{(n+1)(1-\frac zw)}.
$$
For each $\delta\in(0,2)$, we can split the above integral:
$$
\frac1{n+1}\sum_{k=0}^n|\widehat{\mu}(k)|^2=\int_{A_\delta}h_n(z,w)\,d(\mu\times\mu)(z,w)+
\int_{B_\delta}h_n(z,w)\,d(\mu\times\mu)(z,w),
$$
where $A_\delta=\{(z,w)\in\T^2:|z-w|<\delta\}$ and $B_\delta=\{(z,w)\in\T^2:|z-w|\geq\delta\}$.
Note that $|1-(z/w)|\geq \delta$ and $|1-(z/w)^{n+1}|\leq 2$ for $(z,w)\in B_\delta$. Hence $|h_n(z,w)|\leq \frac{2}{\delta(n+1)}$ for $(z,w)\in B_\delta$.
Thus
$$
\biggl|\int_{B_\delta}h_n(z,w)\,d(\mu\times\mu)(z,w)\biggr|\leq \frac{2}{\delta(n+1)}(\mu\times\mu)(B_\delta)\leq \frac{2\|\mu\|^2}{\delta(n+1)}.
$$
On the other hand, for every $z,w\in\T$ and $n\in\N$, $h(z,w)$, being the average of several unimodular numbers, satisfies $|h_n(z,w)|\leq 1$. Hence
$$
\biggl|\int_{A_\delta}h_n(z,w)\,d(\mu\times\mu)(z,w)\biggr|\leq (\mu\times\mu)(A_\delta).
$$
Combining the last three displays, we get
$$
\frac1{n+1}\sum_{k=0}^n|\widehat{\mu}(k)|^2\leq \frac{2\|\mu\|^2}{\delta(n+1)}+(\mu\times\mu)(A_\delta).
$$
Next, note that the continuity of $\mu$ is equivalent to the equality $\lim\limits_{\delta\to0}(\mu\times\mu)(A_\delta)=0$. Thus for every $\epsilon>0$, we can find $\delta\in(0,2)$ such that
$(\mu\times\mu)(A_\delta)<\frac\epsilon2$. Having $\delta$ fixed, we observe that $\frac{2\|\mu\|^2}{\delta(n+1)}<\frac{\epsilon}{2}$ for all sufficiently large $n$. By the above display $\frac1{n+1}\sum\limits_{k=0}^n|\widehat{\mu}(k)|^2<\epsilon$ for all sufficiently large $n$. Since $\epsilon>0$ is arbitrary, this proves (\ref{m0}).
\end{proof}

Recall that a subset $A$ of $\N$ is said to have {\it density} 0 if
$$
\lim_{n\to\infty}\frac{|\{k\in A:k\leq n\}|}{n}=0.
$$
It is straightforward to see that Lemma~\ref{ele} implies the following result.

\begin{corollary}\label{ele0}
Let $\mu\in\M_c(\T)$. Then for every $\epsilon>0$, the set $\{n\in\N:|\widehat{\mu}(n)|\geq\epsilon\}$ has density $0$.
\end{corollary}

\begin{lemma}\label{cme}Let $M$ be a $($finite or$)$ countable subset of $\M_c(\T)$. Then there exists an infinite set $A\subseteq\N$ such that
$$
\lim\limits_{n\to\infty\atop n\in A}\widehat{\mu}(n)=0\ \ \text{for every}\ \ \mu\in M.
$$
\end{lemma}

\begin{proof}Fix an enumeration of $M$: $M=\{\mu_j:j\in\N\}$. It suffices to verify that for every $k\in\N$, the set
$$
A_k=\{n\in\N:|\widehat{\mu_j}(n)|<k^{-1}\ \ \text{for}\ \ 1\leq j\leq k\}
$$
is infinite. Indeed, this being the case we can choose inductively a strictly increasing sequence $\{m_k\}_{k\in\N}$ of positive integers such that $m_k\in A_k$ for every $k\in\N$. By the above display then $|\widehat{\mu_j}(m_k)|<k^{-1}$ for $1\leq j\leq k$. It follows that $\lim\limits_{k\to\infty}\widehat{\mu_j}(m_k)=0$  for every $j\in\N$. That is, the infinite set $A=\{m_k:k\in\N\}$ satisfies the desired conditions.

Thus it remains to verify that each $A_k$ is infinite. It is easy to see that
$$
\N\setminus A_k=\bigcup\limits_{j=1}^k B_{j,k},\ \ \text{where}\ \ B_{j,k}=\{n\in\N:|\widehat{\mu_j}(n)|\geq k^{-1}\}.
$$
By Lemma~\ref{ele0} each $B_{j,k}$ has density $0$. Since the union of finitely many sets of density 0 has density 0, the above display implies that $A_k$ is the complement in $\N$ of a set of density 0. Hence $A_k$ can not possibly be finite, which completes the proof.
\end{proof}

\begin{lemma}\label{uni1} Let $\hh$ be a separable Hilbert space and $U\in L(\hh)$ be an isometric invertible operator with no non-trivial finite dimensional invariant subspaces. Then for every countable set $N\subseteq \hh\times\hh$ there is an infinite set $A\subseteq\N$ such that
$$
\lim\limits_{n\to\infty\atop n\in A}\langle U^na,b\rangle=0\ \ \text{for each $(a,b)\in N$}.
$$
\end{lemma}

\begin{proof} Passing to the complexification, if necessary, we can, without loss of generality, assume that $\K=\C$. In this case $U$ is a unitary operator whose point spectrum is empty. By the spectral theorem, $U$ is unitarily equivalent to the direct $\ell_2$-sum of a finite or countable set of operators $M_j$ of multiplication by the argument ($M_jf(z)=zf(z)$) on $L^2(\T,\mu_j)$ with $\mu_j\in\M_+(\T)$. Thus we can assume that $\hh$ IS the direct $\ell_2$-sum of $L^2(\T,\mu_j)$ and $U$ IS the direct sum of $M_j$:
$$
\hh=\bigoplus\limits_j L^2(\T,\mu_j),\qquad U=\bigoplus_j M_j.
$$

Since the point spectrum of $U$ is empty, so is the point spectrum of each $M_j$. Hence $\mu_j\in\M_c(\T)$ for every $j$. Now for every $(a,b)\in N$ and every $n\in\N$,
$$
\langle U^na,b\rangle=\sum_j\langle M_j^na_j,b_j\rangle=\sum_j\int_\T a_l(z)z^n\overline{b_l(z)}\,d\mu_l(z)=\int_\T z^n\,d\nu_{a,b}(z)=\widehat{\nu_{a,b}}(n),
$$
where $\nu_{a,b}=\sum\limits_j a_j\overline{b_j}\mu_j$ and $h\mu$ stands for the measure absolutely continuous with respect to $\mu\in\M(\T)$ with the density $h\in L^1(\T,\mu)$. Note that the series in the definition of $\nu_{a,b}$ converges with respect to the variation norm since $\|a_j\overline{b_j}\mu_j\|\leq \|a_j\|_{L^2(\mu_j)}\|b_j\|_{L^2(\mu_j)}$ and $\sum\limits_j\|a_j\|_{L^2(\mu_j)}\|b_j\|_{L^2(\mu_j)}\leq \frac12(\|a\|^2+\|b\|^2)$. Furthermore, each $\nu_{a,b}$ is continuous since $\M_c(\T)$ is a closed subspace of the Banach space $\M(\T)$ and each $a_j\overline{b_j}\mu_j$ is continuous.
By Lemma~\ref{cme}, there is an infinite subset $A$ of $\N$ such that $\lim\limits_{n\to\infty\atop n\in A}\widehat{\nu_{a,b}}(n)=0$ for every $(a,b)\in N$. Since $\widehat{\nu_{a,b}}(n)=\langle U^na,b\rangle$, it follows that $\lim\limits_{n\to\infty\atop n\in A}\langle U^na,b\rangle=0$ for each $(a,b)\in N$, as required.
\end{proof}

We can further generalize the above lemma by dropping the invertibility condition.

\begin{lemma}\label{uni2} Let $\hh$ be a separable Hilbert space and $U\in L(\hh)$ be an isometric linear operator with no non-trivial finite dimensional invariant subspaces. Then for every countable set $N\subseteq \hh\times\hh$ there is an infinite set $A\subseteq\N$ such that
$$
\lim\limits_{n\to\infty\atop n\in A}\langle U^na,b\rangle=0\ \ \text{for each $(a,b)\in N$}.
$$
\end{lemma}

\begin{proof} For each $k\in\N$ let $\hh_k$ be the orthogonal complement of $T^k(\hh)$ in $T^{k-1}(\hh)$. We also denote
$\hh_0=\bigcap\limits_{n=0}^\infty U^n(\hh)$. Now it is easy to see that $\hh$ is the direct $\ell_2$-sum of the spaces $\hh_j$ for $j\in\Z_+$. Moreover, $\hh_0$ is a $U$-invariant subspace and the restriction $U_0=U\bigr|_{\hh_0}:\hh_0\to\hh_0$ is an invertible isometry on $\hh_0$. Furthermore, the orthogonal complement $\hh_+$ of $\hh_0$ in $\hh$ (=the direct $\ell_2$-sum of $\hh_j$ for $j\in\N$) is also $U$-invariant, and the restriction $U_+=U\bigr|_{\hh_+}:\hh_+\to\hh_+$ is a forward shift in the sense that $U_+$ provides an invertible isometry between $\hh_j$ and $\hh_{j+1}$ for every $j\in\N$. It immediately follows that $U_+$ is weakly nullifying. That is, $\{U_+^ny\}$ converges weakly to $0$ for each $y\in\hh_+$.

Since $\hh$ is the orthogonal direct sum of $\hh_0$ and $\hh_+$ we can consider the orthogonal projections $P_0$ onto $\hh_0$ along $\hh_+$ and $P_+=I-P_0$ onto $\hh_+$ along $\hh_0$. Let
$$
N_0=\bigcup\limits_{(a,b)\in N}\{P_0a,P_0b\}\ \ \text{and}\ \ N_+=\bigcup\limits_{(a,b)\in N}\{P_+a,P_+b\}.
$$
Since $N_0\times N_0$ is a countable subset of $\hh_0\times\hh_0$ and $U_0$ is  an isometric invertible operator on $\hh_0$ with no non-trivial finite dimensional invariant subspaces, Lemma~\ref{uni1} provides an infinite subset $A$ of $\N$ such that
$$
\lim\limits_{n\to\infty\atop n\in A}\langle U^na_0,b_0\rangle=0\ \ \text{for each $a_0,b_0\in N_0$}.
$$
Since $U_+$ is weakly nullifying,
$$
\lim\limits_{n\to\infty}\langle U^na_+,b_+\rangle=0\ \ \text{for each $a_+,b_+\in N_+$}.
$$
The required property follows easily from the above two displays.
\end{proof}

\subsection{Proof of Theorem~\ref{qq}}

Since $X$ is separable and $F$ is dense in $X$, we can pick a countable set $\{u_{k,0}:k\in\N\}\subset F$, which is norm dense in $X$. By (B3), we can find $\{u_{k,-n}:k,n\in\N\}\subset X$ such that $Tu_{k,-n}=u_{k,1-n}$ for every $k,n\in\N$ and $\lim\limits_{n\to\infty}\|u_{k,-n}\|=0$ for each $k\in\N$. Now we set $u_{k,n}=T^nu_{k,0}$ for $k,n\in\N$. By (B2), the continuous extension $U$ of the restriction $S=T\big|_F:F\to F$ to the $($Hilbert space$)$ completion $\hh$ of the normed space $(E,\|\cdot\|_0)$ is an isometry with no non-trivial finite dimensional invariant subspaces. Applying Lemma~\ref{uni2} with $N=Q\times Q$ and $Q=\{u_{a,b}:a\in\N,b\in\Z_+\}$, we see that there exists an infinite set $A\subseteq\N$ such that
$$
\lim\limits_{n\to\infty\atop n\in A} \langle U^nu_{k,m},u_{a,b}\rangle=0\ \ \text{for every $k,a\in\N$ and $m,b\in\Z_+$}.
$$
Since the restriction of $U$ to the space $F$ containing $E=\spann\{u_{k,n}:n\in\Z_+,\,k\in\N\}$ coincides with the restriction of $T$, we have $u_{k,m+n}=U^nu_{k,m}$ and therefore
$$
\lim\limits_{n\to\infty\atop n\in A} \langle u_{k,m+n},u_{a,b}\rangle=0\ \ \text{for every $k,a\in\N$ and $m,b\in\Z_+$}.
$$
By Theorem~\ref{tt0}, the proof will be complete if we show that the double sequence $\{u_{k,n}\}_{k\in\N,\,n\in\Z}$ satisfies conditions (A1--A6), where we equip $E=\spann\{u_{k,n}:n\in\Z_+,\,k\in\N\}$ with the inner product inherited from $F$. Obviously, (A1), (A2) and (A3) are satisfied.
The already established equality $\lim\limits_{n\to\infty}\|u_{k,-n}\|=0$ implies (A4).
Boundedness of $\{\|u_{k,n}\|_0\}_{n\in\Z_+}$ follows from (A2) and the fact that $T$ acts isometrically on $(E,\|\cdot\|_0)$. Let $k,a\in\N$, $b\in\Z_+$ and $s\in A$. Then using (A2) and the fact that $T$ acts isometrically on $(E,\|\cdot\|_0)$, we get $\langle u_{k,n-s},u_{a,b}\rangle=\langle u_{k,n},u_{a,b+s}\rangle$. Now by the above display
$\lim\limits_{n\to\infty\atop n\in A}\langle u_{k,n-s},u_{a,b}\rangle=0$, which completes the proof of (A5). Finally, let $k,m\in\N$ and $s\in A$.  Using (A2) and the isometry property of $T$ once again, we see that $\langle u_{k,n-s},u_{m,n-t}\rangle=\langle u_{k,t-s},u_{m,0}\rangle$. Hence
$$
\lim\limits_{t\to\infty\atop t\in A,\,t>s}\slim\limits_{n\to\infty\atop n\in A,\,n>t}|\langle u_{k,n-s},u_{m,n-t}\rangle|=\lim\limits_{t\to\infty\atop t\in A,\,t>s}
|\langle u_{k,t-s},u_{m,0}\rangle|=0,
$$
where the last equality holds by the already proven property (A5). Thus (A6) is also verified. That is, (A1--A6) are all satisfied and $T$ is weakly hypercyclic by Theorem~\ref{tt0}.

\section{Further comments and open questions \label{s5}}

We start this last section by proving Proposition~\ref{nns}.
We need the following elementary lemma.

\begin{lemma}\label{suhy} Let $X$ be a Banach space, $S\in L(X)$ and $x\in X$ be a hypercyclic vector for $S$. Assume also $m\in\N$ that $w\in\K$ is such that $w^m=1$ and $w^j\neq 1$ for $1\leq j<m$. Then $u=(x,\dots,x)\in X^m$ is a $1$-weakly hypercyclic vector for $T=S\oplus wS\oplus{\dots}\oplus w^{m-1}S\in L(X^m)$.
\end{lemma}

Note that in the case $\K=\R$, we do not have a wide choice of the numbers $w$. Namely, we must have $w=1$ (trivial) or $w=-1$. The case $w=-1$ is singled out in the next corollary.

\begin{corollary}\label{suhy1} Let $X$ be a Banach space, $S\in L(X)$ and $x\in X$ be a hypercyclic vector for $S$. Then $(x,x)\in X^m$ is a $1$-weakly hypercyclic vector for $T=S\oplus (-S)\in L(X\times X)$.
\end{corollary}

\begin{proof}[Proof of Lemma~$\ref{suhy}$] Let $O=O(T,u)$. Obviously, $O=O_1\cup{\dots}\cup O_m$, where
$$
O_j=\{(S^{mn+j-1}x,w^{j-1}S^{mn+j-1}x,w^{2(j-1)}S^{mn+j-1}x,\dots,w^{(m-1)(j-1)}S^{mn+j-1}x):
n\in\Z_+\}.
$$
Since each of the vectors $x,Sx,\dots,S^{j-1}x$ are hypercyclic for $S$, each of them is also hypercyclic for $S^m$. Indeed, due to Ansari \cite{ansa}, any operator shares its hypercyclic vectors with each of its powers. It follows that each $O_j$ is a dense subset of the closed linear subspace $Y_j$ of $X^m$ defined by the formula
$$
Y_j=\{(y,w^{j-1}y,w^{2(j-1)}y,\dots,w^{(m-1)(j-1)}y):
y\in X\}.
$$
Now the invertibility of the Van-der-Monde matrix $\{w^{(k-1)(j-1)}\}_{j,k=1}^m$
easily implies that
$$
X^m=Y_1\oplus Y_2\oplus{\dots}\oplus Y_m.
$$

Let now $F$ be a non-zero continuous linear functional $X^m$. According to the last display, there is $j\in\{1,\dots,m\}$ such that the restriction of $F_j$ of $F$ to $Y_j$ is non-zero. Hence $F_j$ is a continuous map from $Y_j$ onto $\K$. Since $O_j$ is dense in $Y_j$ it follows that $F(O_j)=F_j(O_j)$ is also dense in $\K$. Since $O_j\subset O$, $F(O)$ is dense in $\K$. By definition, $u$ is a 1-weakly hypercyclic vector for $T$.
\end{proof}

\begin{proof}[Proof of Proposition~$\ref{nns}$] Bayart and Matheron \cite{bm}
constructed a continuous linear operator $S$ on a separable infinite dimensional Hilbert space $\hh$ (real or complex) such that $S$ is hypercyclic, while the operator $S\oplus S\in L(\hh\times\hh)$ is non-cyclic. Consider the operator
$$
T=S\oplus (-S)\in L(\hh\times\hh).
$$
That is, $T(u,v)=(Su,-Sv)$ for every $u,v\in\hh$. By Corollary~\ref{suhy1}, $T$ is 1-weakly hypercyclic. On the other hand, $T^2=S^2\oplus S^2=(S\oplus S)^2$. Thus the non-cyclicity of  $S\oplus S$ implies that $T^2$ is non-cyclic.
\end{proof}

Theorems~\ref{toe01} and~\ref{toe02} leave the following gap.

\begin{question}\label{q2}
Let $g\in H^\infty(\D)$ be such that $|g|>1$ almost everywhere on $\T$ and $\log(|g|-1)\notin L^1(\T)$. Is it possible for $T^*_g\in L(H^2(\D))$ to be weakly hypercyclic or at least $1$-weakly hypercyclic?
\end{question}

Note that the standard comparing the topologies argument extends Theorem~\ref{toe01} to coanalytic Toeplitz operators acting on $H^p(\D)$ with $1\leq p\leq 2$. Furthermore, the proof of Theorem~\ref{toe02} can be adjusted to work for coanalytic Toeplitz operators acting on every $H^p(\D)$ for $1\leq p<\infty$. In this article we just preferred to stick to the Hilbert space situation. It seems though that our proof of Theorem~\ref{toe01} fails miserably for coanalytic Toeplitz operators acting on $H^p(\D)$ with $p>2$.

\begin{question}\label{qq2}
Is there any $g\in H^\infty(\D)$ with $g(\D)\cap\D=\varnothing$ such that $T^*_g$ is weakly hypercyclic as an operator on $H^p(\D)$ for some $p>2$?
\end{question}

\subsection{$2I+S$ is not 1-weakly hypercyclic if $\|S\|\leq 1$}

It is natural to be curious whether Corollary~\ref{coco} extends Banach space operators. Well, it does!

\begin{theorem}\label{q3}
Let $X$ be a $($real or complex$)$ Banach space, $c>0$ and $S\in L(X)$ be power bounded. Then $T=(c+1)I+cS$ is not $1$-weakly hypercyclic.
\end{theorem}

The key to the proof of the above theorem is the following lemma.

\begin{lemma}\label{l1n}
Let $k\geq 2$ be an integer, $c>0$ and for each $n\in\N$, $f_n$ be the rational function
defined by $f_n(z)=\frac{(1-z)^k}{(1+c-cz)^n}$. Let $\{a_m(f_n)\}_{m\in\Z_+}$ be the sequence of the Taylor coefficients of $f_n$ about $0$: $f_n(z)=\sum\limits_{m=0}^\infty a_m(f_n)z^m$ for $|z|<1+\frac1c$ and $N(n)=\sum\limits_{m=0}^\infty|a_m(f_n)|$ $(N(n)$ is finite since $f_n$ is holomorphic in $(1+c^{-1})\D)$. Then there exists $A=A(k,c)>0$ such that $N(n)\leq An^{\frac{1-k}2}$ for every $n\in\N$.
\end{lemma}

It is worth noting that the estimate $N(n)=O(n^{\frac{1-k}2})$ is not the best possible. The best at the level of $O$-estimates is $N(n)=O(n^{-\frac{k}2})$ but its proof we have at the moment is too technical and is left out since the above lemma suffices for our purposes.

\begin{proof}[Proof of Lemma~$\ref{l1n}$] Consider the contour $\Gamma$ given by $\gamma:[-\pi,\pi]\to\C$, $\gamma(t)=2e^{it}-1$. Clearly, $\Gamma$ encircles the origin once in the positive direction and leaves the only pole $1+\frac1c$ of each $f_n$ outside the domain it bounds. By the Cauchy formula,
$$
a_m(f_n)=\frac1{2\pi i}\int_\Gamma \frac{f_n(z)}{z^{m+1}}dz\ \ \ \text{for every $m\in\Z_+$ and $n\in\N$.}
$$
Plugging in the definitions of $f_n$ and $\Gamma$, we get
$$
a_m(f_n)=\frac{2^k}{\pi}\int_{-\pi}^\pi \frac{(1-e^{it})^ke^{it}\,dt}{(1+2c-2ce^{it})^n(2e^{it}-1)^{m+1}}\ \ \ \text{for every $m\in\Z_+$ and $n\in\N$.}
$$
The above display yields
$$
N(n)=\sum_{m=0}^\infty |a_m(f_n)|\leq \frac{2^k}{\pi}\int_{-\pi}^\pi \frac{|1-e^{it}|^k}{|1+2c-2ce^{it}|^n}\sum_{m=0}^\infty|2e^{it}-1|^{-m-1}\,dt\ \ \ \text{for every $n\in\N$.}
$$
Summing up the geometric series and taking into account that the function we integrate is even, we get
$$
N(n)\leq \frac{2^{k+1}}{\pi}\int_{0}^\pi \frac{|1-e^{it}|^k\,dt}{|1+2c-2ce^{it}|^n(|2e^{it}-1|-1)}\ \ \ \text{for every $n\in\N$.}
$$
It is easy to verify that $|1-e^{it}|\leq t$, $|2e^{it}-1|-1\geq \frac{2t^2}{\pi^2}$ and $|1+2c-2ce^{it}|\geq e^{\alpha t^2}$ for every $t\in[0,\pi]$, where $\alpha=\frac{\log(1+4c)}{\pi^2}>0$. Plugging these estimates into above display, we arrive to
$$
N(n)\leq 2^{k}\pi\int_{0}^\pi \frac{t^{k-2}\,dt}{e^{\alpha nt^2}}\ \ \ \text{for every $n\in\N$.}
$$
Performing the change of variables $x=\alpha nt^2$, we obtain
$$
N(n)\leq \frac{2^{k-1}\pi}{(\alpha n)^{\frac{k-1}2}}\int_{0}^{\alpha n\pi^2} x^{\frac{k-3}2}e^{-x}\,dx< n^{\frac{1-k}2}\frac{2^{k-1}\pi}{\alpha^{\frac{k-1}2}}
\int_{0}^{\infty} x^{\frac{k-3}2}e^{-x}\,dx
\ \ \ \text{for every $n\in\N$.}
$$
Since $k\geq 2$, the last integral is finite and the required estimate follows from the above display.
\end{proof}

\begin{corollary}\label{cl1}
Let $S$ be a power bounded operator on a Banach space $X$, $c>0$ and $T=(1+c)I-cS$. Then for every integer $k\geq 2$, there is $B=B(c,k,S)>0$ such that $\|(I-S)^kT^{-n}\|\leq Bn^{\frac{1-k}{2}}$ for every $n\in\N$.
\end{corollary}

\begin{proof} Power boundedness of $S$ easily implies the invertibility of $T$. Let $f_n$, $a_m(f_n)$ and $N(n)$ be as in Lemma~\ref{l1n}. Since $N(n)=\sum\limits_{m=0}^\infty |a_m(f_n)|<\infty$ and $S$ is power bounded, the series $R_n=\sum\limits_{m=0}^\infty a_m(f)S^m$ is operator norm convergent and $\|R_n\|\leq qN(n)$, where $q=\sup\{\|S^m\|:m\in\Z_+\}$. On the other hand, from the equalities $f_n(z)=\frac{(1-z)^k}{(1+c-cz)^n}=\sum\limits_{m=0}^\infty a_m(f_n)z^m$ it easily follows that $R_n=(I-S)^kT^{-n}$. By Lemma~\ref{l1n}, $\|(I-S)^kT^{-n}\|=\|R_n\|\leq qN(n)\leq qA(k,c)n^{\frac{1-k}{2}}$ for every $n\in\N$, which is the required estimate.
\end{proof}

\begin{lemma}\label{pbb} Let $S$ be a power bounded operator on a Banach space $X$, $c>0$ and $T=(1+c)I-cS$. Then for every integer $x\in X$ at least one of the following statements holds$:$
\begin{itemize}
\item $\spann(O(T,x))$ is finite dimensional$;$
\item for every $m>0$, $\lim\limits_{n\to\infty}n^{-m}\|T^nx\|=\infty$.
\end{itemize}
\end{lemma}

\begin{proof} Let $x\in X$. Assume that the second of the above statements fails. That is, there is $m>0$ such that $n^{-m}\|T^nx\|\not\to\infty$. Pick an integer $k\geq 2$ such that $\frac{k-1}2>m$. By Corollary~\ref{cl1}, there is $B>0$ such that $(I-S)^kT^{-n}\leq Bn^{\frac{1-k}2}$ for each $n\in\N$. Hence $\|(I-S)^kx\|=\|(I-S)^kT^{-n}T^nx\|\leq Bn^{\frac{1-k}2}\|T^nx\|$. That is, $\|T^nx\|\geq \frac1Bn^{\frac{k-1}2}\|(I-S)^kx\|$ for every $n\in\N$. Since $\frac{k-1}2>m$, the last inequality is compatible with $n^{-m}\|T^nx\|\not\to\infty$ only if $(I-S)^kx=0$. Since $T=(1+c)I-cS$ and $c\neq 0$, there is a degree $k$ polynomial $p$ such that $(I-S)^k=p(T)$. Thus $p(T)x=(I-S)^kx=0$ and therefore $\spann(O(T,x))$ is at most $(k+1)$-dimensional.
\end{proof}

\begin{proof}[Proof of Theorem~$\ref{q3}$] Let $x\in X$. Since there are no 1-weakly hypercyclic operators on finite dimensional spaces \cite{fe}, $x$ is not a 1-weakly hypercyclic vector for $T$ if $\spann(O(T,x))$ is finite dimensional. If $\spann(O(T,x))$ is infinite dimensional, Lemma~\ref{pbb} guarantees that $\sum\limits_{n=0}^\infty \|T^nx\|^{-1}<\infty$. By Theorem~B, there is $f\in X^*$ such that $|f(T^nx)|>1$ for every
$n\in\N$ and therefore $x$ is not a 1-weakly hypercyclic vector for $T$. Thus $T$ is not 1-weakly hypercyclic.
\end{proof}

\subsection{Norm hypercyclic Toeplitz operators}

It seems that norm hypercyclic Toeplitz operators have not been completely characterized.

\begin{question}\label{q4}
Characterize norm hypercyclic Toeplitz operators $T_g$ in terms of the symbol $g\in L^\infty(\D)$.
\end{question}

There are obvious obstacles to the hypercyclicity of $T_g$. Since the spectrum of any hypercyclic operator must meet $\T$, $\sigma(T_g)\cap\T\neq\varnothing$ is a necessary condition for the hypercyclicity of $T_g$. Of course, a contraction can never be hypercyclic and therefore the hypercyclicity of $T_g$ implies $\|g\|_\infty>1$. Similarly, an expansion can not be (norm) hypercyclic. It is not immediately clear which Toeplitz operators $T_g$ are expansions (that is, satisfy $\|T_gf\|\geq \|f\|$ for each $f\in H^2(\D)$). It is clear however that $T_g$ is an expansion if the distance from $0$ to the convex hull of the (essential) closure of $g(\T)$ is at least $1$. Thus for a hypercyclic $T_g$ the last distance should be less than $1$. Next, a hyponormal operator can not be norm supercyclic  \cite{bdon}, let alone hypercyclic. Since hyponormal Toeplitz operators have been characterized \cite{cow}, this provides another obstacle to hypercyclicity easily formulated in terms of the symbol. Finally, the point spectrum of the adjoint of a hypercyclic operator is always empty. Since the point spectra of Toeplitz operators have been described in terms of the symbol (\cite{st1,dur} in special cases and \cite{st2} in general), this gives yet another obstacle to hypercyclicity written in terms of the symbol. These observations are nearly enough to characterize 3-diagonal hypercyclic Toeplitz operators.

\begin{proposition}\label{3d} Let $a,b,c\in\C$ and $g:\C\setminus\{0\}\to\C$ be given by $g(z)=\frac{a}{z}+b+cz$. Then the Toeplitz operator $T_g\in L(H^2(\D))$ is hypercyclic if and only if $|a|>|c|$ and $\min\limits_{z\in\T}|g(z)|<1<\max\limits_{z\in\T}|g(z)|$.
\end{proposition}

Apart from the above observations, in order to prove Proposition~\ref{3d} we need the following easy lemma.

\begin{lemma}\label{conn} Let $d\in\C$ and $q:\D\to\C$ be a holomorphic function such that $|d|<1$ and $q(z)=q(d/z)$ for every $z\in A$, where $A$ is a subset of the annulus $W=\{z\in\C:|d|<z<1\}$ with at least one accumulation point in $W$. Then $q$ is constant.
\end{lemma}

\begin{proof} Since the two holomorphic on $W$ functions $z\mapsto q(z)$ and $z\mapsto q(d/z)$ coincide on the set $A$ with an accumulation point in $W$, the uniqueness theorem guarantees that $q(z)=q(d/z)$ for every $z\in W$. Extend $q$ to a function on $\C$ by setting $q(z)=q(d/z)$ for $z\in \C\setminus\overline{\D}$. Using the fact that $q(z)=q(d/z)$ for every $z\in W$, one easily sees that thus extended $q$ is an entire function satisfying $q(z)=q(d/z)$ for every $z\in \C\setminus\{0\}$. In particular, $\lim\limits_{|z|\to\infty}q(z)=q(0)$ and therefore $q$ is bounded. By the Liouville theorem, $q$ is constant.
\end{proof}

\begin{proof}[Proof of Proposition~$\ref{3d}$]  First, a direct computation shows that $T^*_gT_g-T_gT^*_g=(|c|^2-|a|^2)P$, where $P=I-T_zT_z^*$. Since $P\geq 0$, $T_g$ is hyponormal (and therefore non-hypercyclic) if $|a|\leq |c|$.

If $\max\limits_{z\in\T}|g(z)|=\|g\|_\infty\leq 1$, then $T_g$ is a contraction and therefore $T_g$ is not hypercyclic. If $\min\limits_{z\in\T}|g(z)|\geq1$, then the distance from $0$ to the convex span of $g(\T)$ is at least $1$ and therefore $T_g$ is an expansion and hence non-hypercyclic.

It remains to prove the hypercyclicity of $T_g$ in the case when $|a|>|c|$ and
$\min\limits_{z\in\T}|g(z)|<1<\max\limits_{z\in\T}|g(z)|$. In \cite{dur}, the explicit description of the point spectra and the eigenvectors of band (that is, the symbol is a Laurent polynomial) Toeplitz operators are given. In our specific case it is an easy exercise anyway (even without looking into \cite{dur}) to see that for each $z$ from the annulus $W_g=\{z\in\C:|c/a|<|z|<1\}$, $T_g f_z=g(z)f_z$, where
$f_z(w)=\sum\limits_{n=0}^\infty (z^{n+1}-(c/(az))^{n+1})w^n$ if $z^2\neq c/a$ and $f_z(w)=\sum\limits_{n=0}^\infty (n+1)z^{n}w^n$ for the 2 points $z$ satisfying $z^2=c/a$. Furthermore, we have just listed all the eigenvectors of $T_g$ up to a scalar multiple (and most of them twice: $f_z=-f_{c/(az)}$ if $z^2\neq c/a$). Thus the point spectrum $\sigma_p(T_g)$ is $\Omega_g=g(W_g)$, which is exactly the region encircled by the ellipse $g(\T)$. The condition $\min\limits_{z\in\T}|g(z)|<1<\max\limits_{z\in\T}|g(z)|$ means that $\Omega_g$ meets both $\D$ and $\C\setminus\overline{\D}$. Let $E_-=\spann\{f_z:|g(z)|<1\}$ and $E_+=\spann\{f_z:|g(z)|>1\}$. Since $E_-$ is spanned by eigenvectors with eigenvalues of absolute value $<1$ and $E_+$ is spanned by eigenvectors with eigenvalues of absolute value $>1$, every $f\in E_-$ has norm convergent to $0$ forward $T_g$-orbit, while every $f\in E_+$ has a norm convergent to $0$ backward $T_g$-orbit. The Kitai Criterion \cite{bama-book} says that $T_g$ must be hypercyclic (even mixing) provided both $E_+$ and $E_-$ are dense in $H^2(\D)$. The latter will be verified if we show that $\spann\{f_z:z\in A\}$ is dense in $H^2(\D)$ for every $A\subseteq W_g$ with at least one accumulation point in $W_g$. Assume the contrary. Then there is an $A\subseteq W_g$ with an accumulation point in $W_g$ such that $\spann\{f_z:z\in A\}$ is not dense in $H^2(\D)$. Without loss of generality $A$ does not contain the points $z$ satisfying $z^2=c/a$.
The non-density of $\spann\{f_z:z\in A\}$ means that there is a non-zero function $h(w)=\sum\limits_{n=0}^\infty h_nw^n$ in $H^2(\D)$ such that $\sum\limits_{n=0}^\infty h_n(f_z)_n=0$ for every $z\in A$. Since
$(f_z)_n=z^{n+1}-(c/(az))^{n+1}$, the last equality is equivalent to $q(z)=q(c/(az))$ for each $z\in A$, where $q(w)=wh(w)$. By Lemma~\ref{conn}, $q$ is constant. Since $q(0)=0$, $q$ is identically $0$ and therefore so is $h$. We have arrived to a contradiction, which completes the proof.
\end{proof}

\subsection{Proof of Propositions~\ref{ninf} and~\ref{ninf1}}

Of course, the hypercyclicity/supercyclicity, the weak hypercyclicity/supercyclicity and the $n$-weak hypercyclicity/supercyclicity all make sense for operators on arbitrary locally convex topological vector spaces. In this section all topological spaces {\it are assumed to be Hausdorff}. Recall that a Fr\'echet space is a complete metrizable locally convex topological vector space. We need the following general concept.

Let $X$ and $Y$ be topological spaces and $\F=\{T_a:a\in A\}$ be a
family of continuous maps from $X$ to $Y$. An element $x\in X$ is
called {\it universal} for $\F$ if the orbit $\{T_ax:a\in A\}$ is
dense in $Y$. We denote the set of universal elements for $\F$
by the symbol $\uu(\F)$. The following lemma follows directly from \cite[Proposition~2.2]{sss}. 

\begin{lemma}\label{gc2} Let $A$ be a set and $X,Y,\Omega$ be
topological spaces such that $Y$ is second countable
and $\Omega$ is compact. For each $a\in A$, let $(\omega,x)\mapsto F_{a,\omega}x$
be a continuous map from $\Omega\times X$ to $Y$. For each $\omega\in\Omega$, let
$\F_\omega=\{F_{a,\omega}:a\in A\}$ be treated as a family of continuous maps from $X$ to $Y$.
Then $\smash{\bigcap\limits_{\omega\in\Omega}\uu(\F_\omega)}$ is a $G_\delta$-subset of $X$.
\end{lemma}

Let $X$ be an infinite dimensional locally convex topological vector space, $n\in\N$ and $T$ be a continuous linear operator on $X$. For the sake of brevity we denote the set of $n$-weakly hypercyclic vectors for $T$ by the symbol $H_n(T)$ and we denote the set of $n$-weakly supercyclic vectors for $T$ by $S_n(T)$. Obviously,
$$
H(T)=\bigcap_{n=1}^\infty H_n(T)\ \ \ \text{and}\ \ \ S(T)=\bigcap_{n=1}^\infty S_n(T)
$$
are the sets of weakly hypercyclic and weakly supercyclic vectors for $T$ respectively. The following lemma is a straightforward modification of a very well-known result on the sets of hypercyclic/supercyclic vector with the proof modified accordingly.

\begin{lemma}\label{lll1} Let $X$ be an infinite dimensional locally convex topological vector space, $n\in\N$ and $T$ be a continuous linear operator on $X$. Then the set $H_n(T)$ is either empty or is dense in $X$. Similarly, $S_n(T)$ is either empty or is dense in $X$ if $n\geq 2$.
\end{lemma}

\begin{proof} First, assume that $x\in H_n(T)$. Then $T$ is $n$-weakly hypercyclic and therefore $T$ is $1$-weakly hypercyclic. According to Feldman \cite{fe}, a $1$-weakly hypercyclic operator can not have non-trivial closed invariant subspaces of finite codimension. It follows that $p(T)$ has dense range for every non-zero polynomial $p$. Let $S:X\to\K^n$ be a surjective continuous linear operator. Then $S(O(T,p(T)x))=S(p(T)(O(T,x)))=(Sp(T))(O(T,x))$. Since $p(T)$ has dense range, $Sp(T):X\to\K^n$ is surjective and therefore $(Sp(T))(O(T,x))$ is dense in $\K^n$ because $x\in H_n(T)$. In particular, $S(O(T,p(T)x))$ is dense in $\K^n$, which proves that $p(T)x\in H_n(T)$. Hence $\{p(T)x:p\neq 0\}\subseteq H_n(T)$. Since every $1$-weakly hypercyclic vector is cyclic, $H_n(T)$ is dense in $X$.

Next, assume that $n\geq 2$ and $x\in S_n(T)$. Then $T$ is $n$-weakly supercyclic and therefore $T$ is $2$-weakly supercyclic. According to Feldman \cite{fe}, in the case $\K=\C$, $T$ can not have closed invariant subspaces of finite codimension $\geq2$. Thus either $p(T)$ has dense range for every non-zero polynomial or there is $\lambda\in\C$ such that $p(T)$ has dense range whenever $p(\lambda)\neq 0$. Running a similar argument in the case $\K=\R$, one easily sees that either $p(T)$ has dense range for every non-zero polynomial or there is an irreducible polynomial $r$ (of degree at most 2 automatically) such that $p(T)$ has dense range whenever $r$ does not divide $p$. Exactly as in the first part of the proof, one sees that $p(T)x\in S_n(T)$ provided $p(T)$ has dense range. Combining these observations with the fact that $x$ is cyclic for $T$, we see that $S_n(T)$ is dense in $X$.
\end{proof}

Recall that if $X$ is a locally convex topological vector space, then the {\it $*$-weak topology} on its dual $X^*$ is the weakest topology making continuous every functional $f\mapsto f(x)$ with $x\in X$. We shall say that a topological space $\Omega$ is a $K_\Sigma$-space if $\Omega$ is the union of countably many of its compact metrizable subspaces. Since an open subset of a metrizable compact space is a $K_\Sigma$-space, we easily see that 
\begin{itemize}
\item[(K1)] An open subset of a $K_\Sigma$-space is a $K_\Sigma$-space;
\item[(K2)] The product of finitely many $K_\Sigma$-spaces is a $K_\Sigma$-space.
\end{itemize}
By a way of warning, note that the product of countably many $K_\Sigma$-spaces may fail to be a $K_\Sigma$-space.

\begin{lemma}\label{wst}Let $X$ be a separable metrizable locally convex topological vector space. Then its dual $X^*$ equipped with the $*$-weak topology is a $K_\Sigma$-space.
\end{lemma}

\begin{proof} Since $X$ is metrizable, we can pick a countable base $\{U_n\}_{n\in\N}$ of neighborhoods of $0$ in $X$. For each $n\in\N$, set $U_n^{\circ}=\{f\in X^*:|f_n(x)|\leq 1\ \ \text{for every $x\in U_n$}\}$. Clearly, $X=\bigcup\limits_{n=1}^\infty U_n^{\circ}$. By the Alaoglu theorem \cite{rr}, each $U_n^{\circ}$ is compact in the $*$-weak topology. Next, since $X$ is separable, there is a dense in $X$ countable set $A$. Then the functionals $\Phi_x(f)=f(x)$  for $x\in A$ are $*$-weakly continuous on $X^*$ and separate points of $X^*$. Since every compact topological space admitting a countable separating points collection of continuous functions is actually metrizable, each $U_n^{\circ}$ is a metrizable compact space. Thus $X^*$ equipped with the $*$-weak topology is a $K_\Sigma$-space.
\end{proof} 

For a locally convex topological vector space $X$, we denote
$$
\Omega_n(X)=\{(f_1,\dots,f_n)\in(X^*)^n:\text{$f_1,\dots,f_n$ are linearly independent}\}.
$$
It is easy to see that $\Omega_n(X)$ is an open subset of $(X^*)^n$ when $X^*$ is equipped with the $*$-weak topology. Since the class of $K_\Sigma$-spaces is closed under finite products and under passing to open subspaces, Lemma~\ref{wst} implies the following corollary.

\begin{corollary}\label{wst1}Let $n\in\N$ and $X$ be a separable metrizable locally convex topological vector space. Then $\Omega_n(X)$ is a $K_\Sigma$-space provided $X^*$ is equipped with the $*$-weak topology.
\end{corollary}

In general, the sequential continuity in the following lemma can not be replaced by the continuity. 

\begin{lemma}\label{freco} Let $X$ be a Fr\'echet space. Then the evaluation map $b:X\times X^*$, $b(x,f)=f(x)$ is sequentially continuous provided $X^*$ carries the $*$-weak topology.
\end{lemma}

\begin{proof} Since $b$ is bilinear, it suffices to show that $b(x_n,f_n)=f_n(x_n)\to 0$ whenever $x_n\to 0$ in $X$ and $f_n\to 0$ in $X^*$ equipped with the $*$-weak topology. By the uniform boundedness principle (=the Banach--Steinhaus theorem) \cite{rr}, the sequence $\{f_n\}$ is uniformly equicontinuous. That is, there is a continuous seminorm $p$ on $X$ such that $|f_n(x)|\leq p(x)$ for every $n$ and every $x\in X$. Hence $|f_n(x_n)|\leq p(x_n)\to 0$ since $x_n\to 0$ in $X$. 
\end{proof}

Now we are ready for the main ingredient of our proof.

\begin{lemma}\label{gde} Let $X$ be a separable Fr\'echet space, $n\in\N$ and $T\in L(X)$. Then $H_n(T)$ and $S_n(T)$ are $G_\delta$-subsets of $X$.
\end{lemma}

\begin{proof} For every $m\in\Z_+$, consider the map
$$
\text{$(f,x)\mapsto F_{f,m}(x)=(f_1(T^mx),\dots,f_n(T^m)x)$ from $\Omega_n(X)\times X$ to $\K^n$.}
$$
According to Lemma~\ref{freco}, this map is sequentially continuous provided $X^*$ is equipped with the $*$-weak topology. By Corollary~\ref{wst1}, $\Omega_n(X)$ is a $K_\Sigma$-space. That is, 
$$
\Omega_n(X)=\bigcup_{j=1}^\infty \Lambda_j,
$$
where each $\Lambda_j$ is a compact metrizable subset of $\Omega_n(X)$. Since the continuity and the sequential continuity coincide for maps defined on a metrizable space, the restriction of the map $(f,x)\mapsto F_{f,m}(x)$ to each $\Lambda_j\times X$ is continuous. For each $f\in \Omega_n(X)$ consider $\F_f=\{F_{f,m}:m\in \Z_+\}$ treated as a family of continuous maps from $X$ to $\K^n$. 
Since $\K^n$ is second countable, Lemma~\ref{gc2} implies that
$$
\text{$\uu_j=\bigcap\limits_{f\in\Lambda_j}\uu(\F_f)$ is a $G_\delta$-subset of $X$ for every $j\in\N$.}
$$
Since the intersection of countably many $G_\delta$-sets is a $G_\delta$-set, the above display yields
$$
\text{$\uu=\bigcap\limits_{f\in\Omega_n(X)}\uu(\F_f)=\bigcap\limits_{j=1}^\infty \uu_j$ is a $G_\delta$-subset of $X$ for every $j\in\N$.}
$$
On the other hand, from the definitions of $\uu$ and $H_n(T)$ one easily sees that
$\uu=H_n(T)$. Thus $H_n(T)$ is a $G_\delta$-subset of $X$.

The supercyclicity part is very similar. For every $m\in\Z_+$ and $w\in\K$, we consider the map
$$
\text{$(f,x)\mapsto F_{f,m,w}(x)=(wf_1(T^mx),\dots,wf_n(T^m)x)$ from $\Omega_n(X)\times X$ to $\K^n$.}
$$
Again, this map is continuous on each $\Lambda_j\times X$. For each $f\in \Omega_n(X)$, we consider $\F'_f=\{F_{f,m,w}:m\in \Z_+,\,w\in\K\}$ treated as a family of continuous maps from $X$ to $\K^n$. As in the first part of the proof, Lemma~\ref{gc2} implies that
$$
\text{$\uu^+_j=\bigcap\limits_{f\in\Lambda_j}\uu(\F'_f)$ is a $G_\delta$-subset of $X$.}
$$
Since the intersection of countably many $G_\delta$-sets is a $G_\delta$-set, the above display yields
$$
\text{$\uu^+=\bigcap\limits_{f\in\Omega_n(X)}\uu(\F'_f)=\bigcap\limits_{j=1}^\infty \uu^+_j$ is a $G_\delta$-subset of $X$.}
$$
From the definitions of $\uu^+$ and $S_n(T)$ it follows that
$\uu^+=S_n(T)$. Thus $S_n(T)$ is a $G_\delta$-subset of $X$. 
\end{proof}

The following result is a generalization of Proposition~\ref{ninf}. 

\begin{proposition}\label{ninf00} Let $X$ be a separable infinite dimensional Fr\'echet space and $T\in L(X)$. Then $T$ is weakly hypercyclic if and only if $T$ is $n$-weakly hypercyclic for each $n\in\N$. Similarly, $T$ is weakly supercyclic if and only if $T$ is $n$-weakly supercyclic for each $n\in\N$.
\end{proposition}

\begin{proof} The 'only if' statements are obvious. Assume that $T$ is $n$-weakly hypercyclic for every $n\in\N$. Then each $H_n(T)$ is non-empty. By Lemma~\ref{lll1}, each $H_n(T)$ is dense in $X$. By Lemma~\ref{gde}, each $H_n(T)$ is a $G_\delta$-set in $X$. By the Baire theorem, the set $H(T)=\bigcap\limits_{n=1}^\infty H_n(T)$ of weakly hypercyclic vectors for $T$ is a dense $G_\delta$-set. In particular, $H(T)$ is non-empty and therefore $T$ is weakly hypercyclic.
The supercyclicity part is similar.
\end{proof}

The following result is a generalization of Proposition~\ref{ninf1}.

\begin{proposition}\label{ninf11} Let $X$ be a separable infinite dimensional Fr\'echet space and $T\in L(X)$. Then the set of weakly hypercyclic vectors for $T$ either is empty or is a dense $G_\delta$ subset of $X$. The same dichotomy holds for the set of weakly supercyclic vectors of $T$.
\end{proposition}

\begin{proof} By Lemmas~\ref{lll1} and~\ref{gde}, either $H_n(T)$ is empty for some $n$ or each $H_n(T)$ is a dense $G_\delta$-subset of $X$. In the first case the set $H(T)$ of weakly hypercyclic vectors for $T$ is empty, while in the second case it is a dense $G_\delta$-subset of $X$ according to the Baire theorem. The supercyclicity part is similar.
\end{proof}

\small

\section{Appendix A: Orbits of expanding coanalytic Toeplitz operators}

Throughout this section for the sake of brevity we employ the following notation:
$$
\E(\D)=\{g\in H^\infty(\D):g(\D)\cap\D=\varnothing\ \ \text{and $|g|>1$ almost everywhere on $\T$}\}.
$$

\begin{lemma}\label{ababa} Let $g\in H^\infty(\D)$ be non-constant and such that
$g(\D)\cap\D=\varnothing$. Then $\|T^*_gf\|>\|f\|$ for every non-zero $f\in H^2(\D)$. In particular, the sequence $\{\|(T^*_g)^nf\|\}_{n\in\Z_+}$ is strictly increasing.
\end{lemma}

\begin{proof} Since
$g(\D)\cap\D=\varnothing$, $|g|\geq 1$ almost everywhere on $\T$. Since $g$ is non-constant, the maximum modulus principle ensures that $|g|>1$ on a subset of $\T$ of positive Lebesgue measure (otherwise $g(\D)\subseteq \T$). By Corollary~\ref{twotoe}, $T_gT^*_g>I$. Hence $\|T^*_gf\|>\|f\|$ for every non-zero $f\in H^2(\D)$.
\end{proof}

The following theorem shows that under the assumptions of Theorem~\ref{toe02}, the orbits of the coanalytic Toeplitz operator grow fast indeed.

\begin{theorem}\label{toe04} Let $g\in \E(\D)$ be such that $\log(|g|-1)\in L^1(\T)$. Then for every $k>0$ and for every non-zero $f\in H^2(\D)$, $\lim\limits_{n\to\infty}n^{-k}\|(T^*_g)^nf\|=+\infty$.
\end{theorem}

\begin{proof} By Lemma~\ref{new}, there is an outer function $h\in H_\infty(\D)$ such that $|h|\leq |g|-1$ almost everywhere on $\T$. Let $n,k\in\N$ be such that $k\leq n$. Then the inequality $(1+|h|)^n\leq |g|^n$ implies $\bin{n}{k}|h^k|\leq |g^n|$ almost everywhere on $\T$.  By Lemma~\ref{twotoe}, $\bin{n}{k}^2T_{h^k}T^*_{h^k}\leq T_{g^n}T^*_{g^n}$. Since $T_{g^n}=T_g^n$, it follows that
$$
\|(T^*_g)^nf\|\geq \bin{n}{k}\|T_{h^k}^*f\|\ \ \text{whenever $n,k\in\N$, $n\geq k$, $f\in H^2(\D)$.}
$$
Since $h$ is outer, the coanalytic Toeplitz operator $T^*_h$ is injective \cite{nik} and therefore $\|T_{h^k}^*f\|=\|(T^*_h)^kf\|>0$ whenever $f\in H^2(\D)$ is non-zero. Hence the above display yields
$$
\ilim_{n\to\infty}\|(T^*_g)^nf\|n^{-k}\geq \frac{\|T_{h^k}^*f\|}{k!}>0\ \ \text{for every $k\in\N$ and every non-zero $f\in H^2(\D)$.}
$$
It immediately follows that $\lim\limits_{n\to\infty}n^{-k}\|(T^*_g)^nf\|=+\infty$ for every $k>0$ and for every non-zero $f\in H^2(\D)$.
\end{proof}

Next, we show that the conclusion of Theorem~\ref{toe04} does not carry on to the whole class $\E(\D)$. That is, the condition $\log(|g|-1)\in L^1(\T)$ can not be removed (not entirely at least). In order to make the result look more spectacular, we introduce the following class of functions:
$$
\E_0(\D)=\{g\in H^\infty(\D):\text{$g(\D)\cap\D=\varnothing$, $g\in C^\infty(\T)$ and for $z\in\T$, $|g(z)|=1\iff z=1$}\}.
$$
Obviously, $\E_0(\D)\subset \E(\D)$.

\begin{remark}\label{rrrrrr} Let $p:\T\to\R$ be such that $p\in C^\infty(\T)$, $p(1)=1$ and $p(z)>1$ for every $z\in\T\setminus\{1\}$. Then the outer function $h$ provided by (\ref{outer}) with $q=\log p$ belongs to $\E_0(\D)$ and satisfies $|h|=p$ on $\T$.
\end{remark}

\begin{theorem}\label{trtrt} Let $q:(0,\infty)\to(0,\infty)$ be any increasing function such that $\lim\limits_{x\to\infty}q(x)=\infty$. Then there exist $g\in \E_0(\D)$ and a non-zero $f\in H^2(\D)$ such that the inequality $\|(T^*_g)^nf\|<q(n)$ holds for infinitely many $n\in\N$.
\end{theorem}

\begin{proof} First, observe that for every increasing $q:(0,\infty)\to(0,\infty)$ satisfying  $\lim\limits_{x\to\infty}q(x)=\infty$, we can find an increasing $\widetilde{q}:(0,\infty)\to(0,\infty)$ such that $\lim\limits_{x\to\infty}\widetilde{q}(x)=\infty$, the function $x\mapsto 2^{-x}\widetilde{q}(x)$ is decreasing and $\widetilde{q}(x)\leq q(x)$ for $x\in [1,\infty)$. It immediately follows that we can, without loss of generality, assume that
\begin{equation}\label{decr}
\text{the function $x\mapsto 2^{-x}q(x)$ is decreasing and $q(1)=1$.}
\end{equation}

For every $\phi\in L^2(\T)$, we consider the continuous linear functional $\Phi_\phi$ on $H^2(\D)$ given by the formula
$$
\Phi_\phi(f)=\int_\T f(z)\phi(z)\,d\lambda(z).
$$
Of course, different functions $\phi$ may lead to the same functional, but we do not care. Obviously,
\begin{equation}\label{estiii}
\|\Phi_\phi\|_{H^2(\D)^*}\leq \|\phi\|_{L^2(\T)}.
\end{equation}

Let $I_n=\{e^{it}:-\frac1n\leq t\leq \frac1n\}$ for $n\in\N$ and $\hh_n$ be the space of $\phi\in L^2(\T)$ vanishing outside $I_n$. First, we observe that for every $n\in\N$,
\begin{equation}\label{estiii0}
\{\Phi_\phi:\phi\in\hh_n\}\ \ \text{is dense in the Banach space $H^2(\D)^*$.}
\end{equation}
Indeed, since $\{\Phi_\phi:\phi\in\hh_n\}$ is a linear subspace of $H^2(\D)^*$ and $H^2(\D)$ is reflexive, in order to verify (\ref{estiii0}), it suffices to show that $\{\Phi_\phi:\phi\in\hh_n\}$ separates the points of $H^2(\D)$. This easily follows from the fact that a non-zero element of $H^2(\D)$ can not vanish almost everywhere on $I_n$.

We shall construct inductively sequences $\{\phi_n\}_{n\in\N}$ in $L^2(\T)$ and $\{k_n\}_{n\in\N}$ of positive integers such that $\|\Phi_{\phi_1}\|=1$ and for every $n\in\N$,
\begin{equation}\label{E14}
\text{$\phi_n\in\hh_n$, $k_{n+1}>k_n$, $\|\phi_n\|_{L^2}\leq \frac{q(k_n)}{4}$ and $\|\Phi_{\phi_{n+1}}-\Phi_{\phi_n}\|\leq 5^{-n}q(k_n)2^{-k_n}$.}
\end{equation}

We start by picking an arbitrary $\phi_1\in \hh_1$ such that $\|\Phi_{\phi_1}\|=1$. Since $q(t)\to\infty$ as $t\to\infty$, we can choose $k_1\in\N$ such that $\|\phi_1\|_{L^2}\leq \frac{q(k_1)}{4}$. The pair $(\phi_1,k_1)$ provides us with the basis of induction. Assume now that $m\in\N$, $m\geq 2$ and $\phi_1,\dots,\phi_{m-1}$ and $k_1,\dots,k_{m-1}$ satisfying (\ref{E14}) are already constructed. By (\ref{estiii0}), we can choose $\phi_m\in\hh_m$ such that $\|\Phi_{\phi_{m}}-\Phi_{\phi_{m-1}}\|\leq 5^{1-m}q(k_{m-1})2^{-k_{m-1}}$. Since $q(t)\to\infty$ as $t\to\infty$, we can choose $k_m\in\N$ such that $k_m>k_{m-1}$ and $\|\phi_m\|_{L^2}\leq \frac{q(k_m)}{4}$. This completes the induction step and thus the inductive construction of $\{\phi_n\}$, $\{k_n\}$ satisfying (\ref{E14}).

By (\ref{E14}) and (\ref{decr}), $\|\Phi_{\phi_{n+1}}-\Phi_{\phi_n}\|\leq 5^{-n}$ for every $n\in\N$. Since also $\|\Phi_{\phi_1}\|=1$, the sequence $\{\Phi_{\phi_n}\}$ converges in the Banach space $H^2(\D)^*$ to a non-zero $\Phi\in H^2(\D)^*$. It is an elementary exercise to show that for every sequence $\{c_n\}_{n\in\N}$ of numbers in $(1,\infty)$, there is a function $p\in C^\infty(\T)$ such that $p(1)=1$, $p(z)>1$ for every $z\in\T\setminus\{1\}$, $\|p\|_{L^\infty(\T)}\leq 2$ and $\|p\|_{L^\infty(I_n)}\leq c_n$ for each $n\in\N$. Indeed, one just has to take $p_n\in C^\infty(\T)$ such that $p_n$ vanishes on $I_n$ and $p_n(z)>0$ for $z\in\T\setminus I_n$ and build $p$ as $p=1+\sum\limits_{k=1}^\infty a_kp_k$. All that remains is to choose the positive coefficients $a_k$ small enough to satisfy all desired conditions. This observation implies that there exists $p\in C^\infty(\T)$ such that $\|p\|_{L^\infty(\T)}\leq 2$,
$\|p\|_{L^\infty(I_n)}\leq 2^{1/k_n}$ for each $n\in\N$, $p(z)>1$ for each $z\in\T\setminus\{1\}$ and
$p(1)=1$. By Remark~\ref{rrrrrr}, there is $g\in\E_0(\D)$ such that $|g|=p$ on $\T$. In particular, $g$ satisfies
\begin{equation}\label{F12}
\text{$\|g\|_{L^\infty(\T)}\leq 2$ and $\|g\|_{L^\infty(I_n)}\leq 2^{1/k_n}$ for each $n\in\N$}.
\end{equation}

Let $n\in\N$. By (\ref{decr}), the sequence $\{2^{-k_m}q(k_m)\}$ decreases and therefore (\ref{E14}) ensures that $\|\Phi_{\phi_{m+1}}-\Phi_{\phi_m}\|\leq 5^{-m}q(k_n)2^{-k_n}$ whenever $m\geq n$. Adding up these inequalities, we obtain $\|\Phi-\Phi_{\phi_n}\|\leq q(k_n)2^{-k_n}\sum\limits_{m=n}^\infty 5^{-m}=\frac54 5^{-n}q(k_n)2^{-k_n}$. Since $\|T'_g\|=\|T_g\|=\|g\|_{L^\infty(\T)}\leq 2$, we have $\|(T'_g)^{k_n}(\Phi-\Phi_{\phi_n})\|\leq \frac54 5^{-n}q(k_n)$. Hence
$$
\|(T'_g)^{k_n}(\Phi-\Phi_{\phi_n})\|\leq \frac{q(k_n)}{4}.
$$
Next, it is easy to see that $T'_g\Phi_{\phi}=\Phi_{g\phi}$ for every $\phi\in L^2(\T)$. Combining this observation with (\ref{estiii}), we get $\|(T'_g)^m\Phi_\phi\|\leq \|\phi\|_{L^2}\|g\|^m_{L^\infty(I_n)}$ for each $\phi\in\hh_n$ and $m\in\Z_+$. Applying this to $\Phi_{\phi_n}$ and using (\ref{E14}) and (\ref{F12}), we get
$$
\|(T'_g)^{k_n}\Phi_{\phi_n}\|\leq \|\phi_n\|_{L^2}\|g\|_{L^\infty(I_n)}^{k_n}\leq \frac{q(k_n)}{2}.
$$
Combining the last two displays, we obtain $\|(T'_g)^{k_n}\Phi\|< q(k_n)$ for each $n\in\N$. By
Remark~\ref{rire}, $T'_g$ and $T^*_g$ are isometrically similar. Thus there is a non-zero $f\in H^2(\D)$ for which $\|(T^*_g)^{k_n}f\|\leq q(k_n)$ for every $n\in\N$.
\end{proof}

\begin{corollary}\label{cococooc} There exist $g\in\E_0(\D)$ and a non-zero $f\in H^2(\D)$ such that
\begin{equation}\label{summa}
\sum_{n=0}^\infty \|(T^*_g)^nf\|^{-c}=\infty\ \ \text{for every $c>0$.}
\end{equation}
\end{corollary}

\begin{proof} By Theorem~\ref{trtrt}, there are $g\in \E_0(\D)$, a non-zero $f\in H^2(\D)$ and a strictly increasing sequence $\{k_n\}_{n\in\N}$ of positive integers such that $k_1\geq 2$ and $\|(T^*_g)^{k_n}f\|\leq \log k_n$ for each $n\in\N$. Now let $c>0$. By Lemma~\ref{ababa}, the sequence $\{\|(T^*_g)^nf\|^{-c}\}$ is decreasing. Hence
$$
\sum_{j=0}^{k_n}\|(T^*_g)^jf\|^{-c}\geq (k_n+1)\|(T^*_g)^{k_n}f\|^{-c}\geq \frac{k_n}{(\log k_n)^c}.
$$
Since the sequence in the right-hand side of the above display converges to infinity, (\ref{summa}) follows.
\end{proof}

Since $\D$ is simply connected, the monodromy theorem implies that every invertible $g\in H^\infty(\D)$ can be written as $g=e^u$, where $u:\D\to\D$ is holomorphic. Moreover, $u$ is determined uniquely up to adding a constant from $2\pi i\Z$. It makes sense to write $u=\log g$. It is well-known (follows for instance, from the well-established properties of the Hilbert transform) that $u$ belongs to $H^p(\D)$ for each $p<\infty$ but may fail to be bounded. We shall need the following technical restriction of the class $\E(\D)$:
$$
\E_1(\D)=\{g\in \E(\D):\text{$\log g\in H^\infty(\D)$ and $\overline{(\log g)(\D)}\cap (i\R)$ is a singleton}\}.
$$
Note that if $g\in \E(\D)$ is continuous on $\T$ and satisfies the H\"older condition $|g(z)-g(w)|\leq c|z-w|^\alpha$ for some $c,\alpha>0$ for all $z,w\in\T$, then $\log g\in H^\infty(\D)$ automatically. It easily follows that $\E_0(\D)\subset \E_1(\D)$. The following result is in a very strong contrast with Theorem~\ref{trtrt}.

\begin{theorem}\label{stest}
Let $g\in \E_1(\D)$ and $f\in H^2(\D)$, $f\neq 0$. Then for every $k>0$, $\slim\limits_{n\to\infty}n^{-k}\|(T^*_g)^nf\|=+\infty$.
\end{theorem}

Note that Theorem~\ref{stest} together with Theorem~\ref{trtrt} demonstrate that orbits of expanding coanalytic operators may exhibit highly irregular behavior. We split the proof into a number of lemmas. From now on, we denote
$$
\Pi=\{z\in\C:{\rm Re}\,z\geq 0\}.
$$
That is $\Pi$ is the closed upper half plane. We shall also use the symbol $\pp$ to denote the space $\C[z]$ of complex polynomials.

\begin{lemma}\label{polyyy} Let $K$ be a compact subset of $\Pi$ such that $K\cap\R=\{0\}$, $r>0$ and $k\in\Z_+$. Then for every complex valued $f\in C^k[-r,r]$ satisfying $f(0)={\dots}=f^{(k)}(0)=0$, there exists a sequence $\{p_n\}_{n\in\N}$ in $\pp$ such that $p_n^{(j)}$ converge to $f^{(j)}$ uniformly on $[-r,r]$ for $0\leq j\leq k$ and $p_n$ converge to $0$ uniformly on $K$.
\end{lemma}

\begin{proof} Replacing $K$ with its closed convex hull, we can without loss of generality assume that $K$ is convex (the hull will still lie in $\Pi$ and will still meet $\R$ only at $0$). Let $Q=K\cup [-r,r]$ and consider the map $g:Q\to\C$ given by the formula
$$
g(z)=\left\{\begin{array}{ll}f^{(k)}(z)&\text{if $z\in[-r,r]$;}\\ 0&\text{if $z\in K$.}\end{array}\right.
$$
Since $f^{(k)}(0)=0$ and $K\cap[-r,r]=\{0\}$, $g$ is well-defined and continuous. Obviously, $g$ vanishes and therefore is holomorphic in the interior of $Q$. Since $K$ is convex and satisfies $K\cap\R=\{0\}$, the open set $\C\setminus Q$ is connected. By the Mergelyan theorem \cite{ru}, there is a sequence $\{q_n\}_{n\in\N}$ of polynomials, such that $q_n$ converge to $g$ uniformly on $Q$. Consider the linear map $V:\pp\to\pp$ defined by
$$
V\Bigl(\sum_{j=0}^m a_jz^j\Bigr)=\sum_{j=0}^m \frac{a_j}{j+1}z^{j+1}.
$$
Equivalently, $Vp(z)=\int_\Gamma p(\xi)\,d\xi$, where $\Gamma$ is any directed smooth path starting at $0$ and terminating at $z$. Clearly, $(Vp)'=p$ for each $p$.

Now we define $p_n=V^kq_n$ for $n\in\N$. Obviously, $\{p_n\}$ is a sequence of polynomials. Since, $q_n$ converge to $f^{(k)}$ uniformly on $[-r,r]$ and $f(0)={\dots}=f^{(k)}(0)=0$, one easily sees that $V^jq_n$ converge to $f^{(k-j)}$ uniformly on $[-r,r]$ for $0\leq j\leq k$. Noticing that $V^jq_n=p_n^{(k-j)}$, we see that $p_n^{(j)}$ converge to $f^{(j)}$ uniformly on $[-r,r]$ for $0\leq j\leq k$.

Finally, using the presentation $Vp(z)=\int_{[0,z]}p(\xi)\,d\xi$ and estimating the integrals in the most straightforward way, one gets
$$
|V^jp(z)|\leq \frac1{(k-1)!}\max\{|p(\xi)|:\xi\in[0,z]\}\ \ \text{for every $z\in\C$ and $p\in\pp$,}
$$
 where $[0,z]$ is the segment with the ends $0$ and $z$. Since $K$ is convex and $0\in K$, it follows that $\|p_n\|_{C(K)}\leq \frac1{(k-1)!}\|q_n\|_{C(K)}$. Thus the uniform convergence of $q_n$ to $0$ on $K$ implies the uniform convergence of $p_n$ to $0$ on $K$.
\end{proof}

\begin{lemma}\label{holyyy} Let $k\in\Z_+$. Then there is a sequence $\{f_n\}_{n\in\N}$ of entire functions uniformly bounded on $\Pi$ and such that $f_n^{(k)}(0)=1$ for every $n\in\N$, $f_{n}^{(j)}(0)\to 0$ for $0\leq j<k$ and $f_n(z)\to 0$ whenever ${\rm Im}\,z>0$.
\end{lemma}

\begin{proof} Just take $f_n(z)=\frac{e^{inz}}{(in)^k}$. \end{proof}

We will use the Paley--Wiener theorem for distributions. It is a good moment to recall the related concepts and facts and introduce the corresponding notation. For background reading on the subject we refer to \cite{ho}.
First, recall that it is said that an entire function $f$ is of {\it exponential type} if there is $c>0$ such that the function $z\mapsto|f(z)|e^{-c|z|}$ is bounded on $\C$. The infimum of all such constants $c$ is called the {\it exponential type} of $f$.

{\it Distributions with compact support} on $\R$ are by definition, the continuous linear functionals on the Fr\'echet space $C^\infty(\R)$, which carries its natural topology of uniform convergence of each derivative on each finite interval. The space of distributions with compact support is traditionally denoted ${\cal E}'(\R)$. It is said that $\psi\in {\cal E}'(\R)$ vanishes on an open interval $I$ if $\psi(f)=0$ for every $f\in C^\infty(\R)$ vanishing outside $I$. The support of $\psi$ is the complement to the union of all intervals on which $\psi$ vanishes. It so happens that the support of every $\psi\in {\cal E}'(\R)$ is compact (thus the name) and is non-empty whenever $\psi\neq0$. The Fourier transform $\widehat{\psi}:\C\to\C$ of $\psi\in{\cal E}'(\R)$ is defined as
$$
\widehat{\psi}(z)=\psi(f_z),\ \ \text{where $f_z(x)=e^{-ixz}$}.
$$

The Paley Wiener theorem characterizes the Fourier transforms of distributions with compact support.

\begin{thmpw} A function $F:\C\to\C$ is the Fourier transform of a distribution with compact support if and only if $F$ is an entire function of exponential type and is polynomially bounded on the real axis, that is, the function $x\mapsto F(x)(1+|x|)^{-k}$  is bounded on $\R$ for some positive $k=k(F)$. Furthermore, for $\psi\in {\cal E}'(\R)$ the exponential type of $\widehat{\psi}$ is exactly the smallest $r>0$ such that $[-r,r]$ contains the support of $\psi$. \end{thmpw}

We mention two other well-known facts. Namely, for each $\psi\in {\cal E}'(\R)$:
\begin{itemize}
\item[(PW1)]If $r$ is the exponential type of $\widehat{\psi}$, then $\psi$ is continuous with respect to the topology of the Fr\'echet space $C^\infty[-r,r]$. Equivalently, $\psi(f)$ depends only on the restriction of $f\in C^\infty(\R)$ to $[-r,r]$.
\item[(PW2)]If $k\in\Z_+$ and $|\widehat{\psi}(x)|(1+|x|)^{2-k}$ is bounded on the real line, then $\psi$ is continuous with respect to the topology of the Fr\'echet space $C^k(\R)$ and therefore admits a unique extension to a continuous linear functional on $C^k(\R)$.
\end{itemize}

\begin{lemma}\label{wei} Let $K$ be a compact subset of $\Pi$ such that $K\cap\R=\{0\}$ and let $\mu$ be a finite complex valued Borel $\sigma$-additive measure on $K$. Then for the entire function
$$
F_\mu(z)=\int_K e^{-iz\xi}\,d\mu(\xi)
$$
exactly one of the following two conditions holds$:$
\begin{itemize}
\item[{\rm(\ref{wei}.1)}]$F_\mu(z)=c$ for every $z\in\C$, where $c=\mu(\{0\});$
\item[{\rm(\ref{wei}.2)}]the function $x\mapsto |F_\mu(x)|(1+|x|)^{-k}$ is unbounded on $\R$ for every $k>0$.
\end{itemize}
\end{lemma}

\begin{proof} Assume that (\ref{wei}.2) fails. We have to show that $F_\mu(z)=c$ for every $z$, where $c=\mu(\{0\})$.
Since (\ref{wei}.2) fails there is $k\in\Z_+$ such that the function $|F_\mu(x)|(1+|x|)^{2-k}$ is bounded on $\R$. Obviously, $F_\mu$ is of exponential type with the type not exceeding $r=\sup\{|z|:z\in K\}$. By Theorem~PW, there is $\theta\in {\cal E}'(\R)$ such that $F_\mu=\widehat\theta$. By (PW2) $\theta$ can be viewed as a continuous linear functional on $C^k(\R)$, while by (PW1) $\theta(f)$ depends only on the restriction of $f$ to $[-r,r]$ for every $f\in C^k(\R)$.
Next the equality $\widehat\theta=F_\mu$ means that $\theta(f_z)=\int_K f_z\,d\mu$ for every $z\in \C$, where $f_z(w)=e^{-izw}$. Since the linear span of $f_z$ is dense in the space of entire functions $\hh(\C)$ and the functionals $\theta$ and $f\mapsto \int_K f \,d\mu$ are continuous on $\hh(\C)$, we have
\begin{equation}\label{thg}
\theta(g)=\int_K g\,d\mu\ \ \text{for every $g\in\hh(\C)$}.
\end{equation}

Now let $f\in C^\infty(\R)$ be such that $f(0)={\dots}=f^{(k)}(0)=0$. By Lemma~\ref{polyyy}, there is a sequence $\{p_n\}_{n\in\N}$ of polynomials such that $p_n^{(j)}$ converge to $f^{(j)}$ uniformly on $[-r,r]$ for $0\leq j\leq k$ and $p_n$ converge to $0$ uniformly on $K$. Since $\theta$ is continuous with respect to the topology of $C^k[-r,r]$, $\theta(p_n)\to\theta(f)$. Since $p_n$ converge to $0$ uniformly on $K$, $\int_K p_n\,d\mu\to 0$. Thus (\ref{thg}) ensures that $\theta(f)=0$. Hence $\theta$ vanishes on each $f\in C^\infty(\R)$ be such that $f(0)={\dots}=f^{(k)}(0)=0$ (such functions form a closed subspace of $C^\infty(\R)$ of codimension $k+1$). It immediately follows that there are constants $c_0,\dots,c_k\in\C$  such that
$\theta(f)=\sum_{j=0}^k c_jf^{(j)}(0)$ for every $f\in C^\infty(\R)$.
Obviously, there is $m\in\{0,\dots,k\}$ such that $c_j=0$ for $m<j\leq k$ and either $c_m\neq 0$ or $m=0$. Then
$$
\theta(f)=\sum_{j=0}^m c_jf^{(j)}(0)\ \ \ \text{for every $f\in C^\infty(\R)$}
$$
(we have just disregarded the last few zeros among $c_j$ if any). By Lemma~\ref{holyyy}, there is a sequence $\{f_n\}_{n\in\N}$ of entire functions uniformly bounded on $\Pi$ and such that $f_n^{(m)}(0)=1$ for every $n$, $f_{n}^{(j)}(0)\to 0$ for $0\leq j<m$ and $f_n(z)\to 0$ whenever ${\rm Im}\,z>0$. According to the above display $\theta(f_n)\to c_m$. Now the dominated convergence theorem guarantees that $\int_K f_n\,d\mu\to 0$ if $m>0$ and $\int_K f_n\,d\mu\to c=\mu(\{0\})$ if $m=0$. These observations together with (\ref{thg}) imply that $m=0$ and $c_0=c$. Hence $\theta(f)=cf(0)$ for every $f\in C^\infty(\R)$. It follows that $F_\mu=\widehat\theta\equiv c$.
\end{proof}

\begin{lemma}\label{brr} Let $h\in H^\infty(\D)$ be such that $\overline{i h(\D)}\subset\Pi$ and $\overline{i h(\D)}\cap \R$ is a singleton. Let also $q,f\in H^2(\D)$ and $F$ be the entire function defined by the formula $F(z)=\langle q,T^*_{e^{zh}}f\rangle$. Then exactly one of the following two conditions holds$:$
\begin{itemize}
\item $F$ is identically $0;$
\item the function $x\mapsto F(x)(1+|x|)^{-k}$ is unbounded on $\R$ for every $k>0$.
\end{itemize}
\end{lemma}

\begin{proof} Note that adding to $h$ a constant from $i\R$ shifts the unique common point of $\overline{i h(\D)}$ and $\R$ along the real line and multiplies $F$ by a function of the form $e^{icz}$ with $c\in\R$ thus not changing $|F|$ on the real line. Hence, without loss of generality, we may assume that $K\cap \R=\{0\}$, where $K=\overline{i h(\D)}$. Define the finite complex valued Borel $\sigma$-additive measure $\mu$ on the compact set $K$ by the formula
$$
\mu(A)=\int_{\{w\in\T:ih(w)\in A\}} q(\xi)\overline{f(\xi)}\,d\lambda(\xi).
$$
Then
$$
F(z)=\int_\T q(\xi)\overline{f(\xi)}e^{zh(\xi)}\,d\lambda(\xi)=\int_K e^{-izw}\,d\mu(w)\ \ \text{for each $z\in\C$.}
$$
Since $\mu(\{0\})=0$, Lemma~\ref{wei} guarantees that either $F$ is identically $0$ or the function $F(x)(1+|x|)^{-k}$ is unbounded on $\R$ for every $k>0$.
\end{proof}

\begin{proof}[Proof of Theorem~$\ref{stest}$] Let $h=\log g$. Since $g\in \E_1(\D)$, $h\in H^\infty(\D)$, $\overline{i h(\D)}\subset\Pi$ and $\overline{i h(\D)}\cap \R$ is a singleton. Assume the contrary. Then there are a non-zero $f\in H^2(\D)$ and $k>0$ such that the sequence $\{\|(T^*_g)^nf\|n^{-k}\}_{n\in\N}$ is bounded. Since $g=e^h$, $(T^*_g)^nf=T^*_{e^{nh}}f$. Since $|e^{sh}|\leq |e^{th}|$ almost everywhere on $\T$ provided $s\leq t$, Lemma~\ref{twotoe} ensures that the function $N:\R\to\R$, $N(t)=\|T^*_{e^{th}}f\|$ is increasing. In particular, $\|T^*_{e^{th}}f\|\leq \|T^*_{e^{nh}}f\|=\|(T^*_g)^nf\|$ for every $n\in\N$ and $t\in(n-1,n]$ and
$\|T^*_{e^{th}}f\|\leq \|f\|$ if $t\leq 0$. Since $\{\|(T^*_g)^nf\|n^{-k}\}_{n\in\N}$ is bounded, it follows that the function $t\mapsto (1+|t|)^kN(t)$ is bounded on $\R$.
Now for every $q\in H^2(\D)$ consider the entire function
$$
F_q(z)=\langle q,T^*_{e^{zh}}f\rangle.
$$
Clearly $|F_q(t)|\leq \|q\|N(t)$ for every $t\in\R$ and therefore the function $F_q(t)(1+|t|)^{-k}$ is bounded on the real line. By Lemma~\ref{brr}, $F_q\equiv0$  for each $q\in H^2(\D)$. In particular, $\langle f,f\rangle=F_f(0)=0$. Hence $f=0$. This contradiction completes the proof.
\end{proof}

We end up this section with the following natural question.

\begin{question}\label{q1}
Does the conclusion of Theorem~$\ref{stest}$ hold for every $g\in\E(\T)$?
\end{question}

Note that the definition of $\E_1(\D)$ has two extra conditions compared with the definition of $\E(\D)$. One, the boundedness of $\log g$ is relatively mild. It is needed to ensure that the entire functions $F_{f,q}(z)=\langle q,T^*_{e^{z\log g}}f\rangle$ are of exponential type (remaining of exponential order $1$, they fail to be of exponential type for generic $f,q\in H^2(\D)$ if $\log g$ is unbounded). Nevertheless, there is a number of promising ways to circumvent this obstacle. The second condition of $\overline{\log g(\D)}$ meeting the imaginary axis at one point (enabling us to apply the Mergelyan theorem) is much worse. Indeed, our line of proof of Theorem~\ref{stest} invariably goes through proving that each non-zero $F_{f,q}$ is not polynomially bounded on the real axis. This very statement fails miserably if $\overline{\log g(\D)}$ meets the imaginary axis at more than 1 point. Thus to solve Question~\ref{q1}, one has to come up with an essentially different idea.

%

\vfill\break

\small\rm

\vskip1truecm

\scshape

\noindent Stanislav Shkarin

\noindent Queens's University Belfast

\noindent Pure Mathematics Research Centre

\noindent University road, Belfast, BT7 1NN, UK

\noindent E-mail address: \qquad {\tt s.shkarin@qub.ac.uk}


\begin{thebibliography}{99}

\itemsep=-2pt

\bibitem{ansa}S.~Ansari, \it Hypercyclic and cyclic vectors, \rm J. Funct.
Anal. \bf128\rm\  (1995), 374--383

\bibitem{ball} K.~Ball, \it The plank problem for symmetric bodies, \rm Invent. Math. \bf104\rm\ (1991), 535--543

\bibitem{ball1} K.~Ball, \it The complex plank problem, \rm
Bull. London Math. Soc. \bf 33\rm\ (2001), 433--442

\bibitem{bm}F.~Bayart and E.~Matheron, \it Hypercyclic operators
failing the Hypercyclicity Criterion on classical Banach spaces, \rm
J. Funct. Anal. \bf250\rm\  (2007), 426--441

\bibitem{bama-book}F.~Bayart and E.~Matheron, \it Dynamics of linear
operators, \rm Cambridge University Press, 2009

\bibitem{bdon}P.~Bourdon, \it Orbits of hyponormal operators, \rm
Michigan Math. J. \bf 44\rm\ (1997), 345--353

\bibitem{berg} P.~Bourdon and J.~Shapiro, \it Hypercyclic operators that commute with the Bergman backward shift, \rm Trans. AMS \bf352\rm\ (2000), 5293--5316

\bibitem{cs}K.~Chan and R.~Sanders, \it A weakly hypercyclic operator that is not norm
hypercyclic, \rm J. Operator Theory {\bf 52} (2004), 39--59

\bibitem{cow}C.~Cowen, \it Hyponormality of Toeplitz Operators, \rm Proc. Amer. Math. Soc. \bf103\rm (1988), 809--812
\bibitem{dur}P.~Duren, \it On the spectrum of a Toeplitz operator, \rm Pacific J. Math. \bf14\rm\ (1964), 21–-29

\bibitem{fe} N.~Feldman, \it $N$-weakly hypercyclic and $N$-weakly supercyclic operators, \rm J. Functional Analysis [to appear]

\bibitem{kc} R.~Gethner and J.~Shapiro, \it Universal vectors for
operators on spaces of holomorphic functions, \rm Proc. Amer. Math.
Soc.  \bf100\rm\  (1987), 281--288

\bibitem{gs}G.~Godefroy and J.~Shapiro, \it Operators with dense,
invariant cyclic vector manifolds, \rm J. Funct. Anal., \bf98\rm\
(1991), 229--269

\bibitem{st1}P.~Hartman and A.~Wintner, \it The spectra of Toeplitz's matrices, \rm Amer. J. Math. \bf76\rm\ (1954), 867–882

\bibitem{ho}L.~H\"ormander, \it Linear Partial Differential Operators, \rm Springer, Berlin, 1976

\bibitem{st2}M.~Krein, \it Integral equations on the half-line with a kernel depending on the difference of the arguments \rm (Russian), Uspehi Mat. Nauk \bf13\rm\ (1958) 3–-120

\bibitem{levin}B.~Levin, \it Distribution of Zeros of
Entire Functions, \rm AMS, Providence, R.I., 1980

\bibitem{nik}N.~Nikolsky, \it Treatise on the shift operator, \rm Springer, Berlin, 1986

\bibitem{rr}A.~Robertson and W.~Robertson, \it Topological vector spaces, \rm Cambridge University Press, New York, 1964

\bibitem{ru}W.~Rudin, \it Real and Complex Analysis, \rm McGraw-Hill, New York, 1987

\bibitem{ss}S.~Shkarin, \it Non-sequential weak supercyclicity and
hypercyclicity, \rm J. Funct. Anal. \bf242\rm\ (2007), 37--77

\bibitem{sss}S.~Shkarin, \it Remarks on common hypercyclic vectors,
\rm Journal of Functional Analysis \bf 258\rm\ (2010), 132--160


\end{thebibliography}
\end{document}